\documentclass[11pt, reqno]{amsart}
\usepackage{ graphicx }
\usepackage[utf8]{ inputenc }
\usepackage{ svg }
\usepackage{ amsmath }
\usepackage{ amsthm }
\usepackage{ amssymb }
\usepackage{ dsfont }
\usepackage{ xargs }
\usepackage{ enumitem }
\usepackage{ esint }
\usepackage{ ragged2e }
\usepackage{ mathrsfs }
\usepackage{ scalerel }
\usepackage{ cancel }
\usepackage{ todonotes }
\usepackage{comment}

\usepackage[hypertexnames = false]{ hyperref }

\theoremstyle{plain}
\newtheorem{theorem}{Theorem}[section]
\newtheorem{proposition}[theorem]{Proposition}
\newtheorem{corollary}[theorem]{Corollary}
\newtheorem{lemma}[theorem]{Lemma}

\theoremstyle{definition}
\newtheorem{definition}[theorem]{Definition}
\newtheorem{remark}[theorem]{Remark}
\newtheorem{example}[theorem]{Example}

\numberwithin{equation}{section}

\newcommand{\RR}{\mathds{R}}

\newcommand{\NN}{\mathds{N}}
\newcommand{\CC}{\mathds{C}}
\newcommand{\GG}{\mathds{G}}
\newcommand{\HH}{\mathds{H}}

\newcommand{\g}{\mathfrak{g}}
\newcommand{\h}{\mathfrak{h}}

\newcommand{\Hormander}{Hörmander}
\newcommand{\Nikodym}{Nikodým}
\newcommand{\Schrodinger}{Schrödinger}
\newcommandx{\Haus}[1][1=s]{\mathscr{H}^{#1}}
\newcommandx{\Leb}[1][1=s]{\mathscr{L}^{#1}}
\newcommandx{\norm}[2][1=,2=]{\| {#1} \|_{#2}}
\newcommand{\loc}{\text{loc}}
\newcommand{\supp}{\operatorname{supp}}
\newcommand{\diverg}{\operatorname{div}}

\newcommand{\Span}{\operatorname{span}}

\def\lesssimA#1#2{\mathrel{\vcenter{\offinterlineskip
			\ialign{\hfil##\hfil\cr\vspace{2pt}\cr$#1<$\cr$#1\sim$\cr}
}}}
\def\lesssim{\mathpalette\lesssimA{}}

\begin{document}
	
	\pagestyle{plain}
	\thispagestyle{empty}
	
	\author{Guido De Philippis}
	\address[G.D.P.]{Dipartimento di Matematica ``T. Levi-Civita'', via Trieste 63, 35121 Padova, Italy.}
	\email{guido.dephilippis@math.unipd.it}
	
	\author{Alberto Gervani}
	\address[A.G.]{Scuola Normale Superiore, Piazza dei Cavalieri 7, 56126 Pisa, Italy \&
    Scuola Galileiana di Studi Superiori,  Via Venezia 20, 35131 Padova, Italy.}
	\email{alberto.gervani@sns.it}
	
	\author{Annalisa Massaccesi}
	\address[A.M.]{Dipartimento di Matematica ``T. Levi-Civita'', via Trieste 63, 35121 Padova, Italy.}
	\email{annalisa.massaccesi@unipd.it}
	
	\author{Davide Vittone}
	\address[D.V.]{Dipartimento di Matematica ``T. Levi-Civita'', via Trieste 63, 35121 Padova, Italy.}
	\email{davide.vittone@unipd.it}
	
	\title{$\mathscr{A}$-free measures and the Rank-one Theorem in Carnot Groups}

    \date{\today}


    \begin{abstract}
        We establish a structure theorem for measures satisfying left-invariant PDE constraints in Carnot groups. Our generalization, that takes inspiration from the paper~\cite{DPR} by the first named author and F. Rindler, requires the introduction of the hypoelliptic wave cone, a sub-Riemannian version of the wave cone from the theory of Compensated Compactness, and it employs classical results in Harmonic Analysis on homogeneous spaces. Then, we utilize this result to extend the Rank-one Theorem for functions of bounded horizontal variation ($BV_H$) to all Carnot groups. This is achieved by considering a suitable curl-type left-invariant operator, that provides necessary differential constraints for horizontal gradients, and by studying the hypoellipticity of said constraints.
    \end{abstract}

	\maketitle
    
	\tableofcontents
	
	\section{Introduction}
    The rank-one property for maps with bounded variation ($BV$) was proved by G.~Alberti~\cite{alberti} by what is generally regarded as a tour-de-force in Geometric Measure Theory. More recently, two independent proofs were found in~\cite{DPR}, where PDEs techniques were employed, and in~\cite{massaccesi-vittone}, by a geometric argument.
    See~\cite{ambrosio-fusco-pallara} for more details on $BV$ functions.
    
    A related question, that arises quite naturally and that was first investigated in~\cite{don-massaccesi-vittone},  concerns the validity of the rank-one property for {\em horizontally BV} maps in  sub-Riemannian geometry, in particular in the setting of {\em Carnot groups}. Recall that a Carnot group $\GG$ is a connected, simply connected and nilpotent Lie group whose Lie algebra $\mathfrak g$ is stratified; in particular, $\mathfrak g$ is bracket-generated by its first {\em horizontal} layer $\mathfrak g_1$, of which we fix a basis $X_1,\dots,X_\rho$. Recall also that $\GG$ can be identified with $\RR^d$ by exponential coordinates, according to which the Lebesgue measure $\mathscr L^d$ is a Haar measure. A  map $u:\Omega\to\RR^h$ defined on an open subset $\Omega\subset\GG$ has bounded horizontal variation ($u\in BV_H(\Omega,\RR^h)$) if its horizontal distributional derivatives  $D_Hu:=(X_1u,\dots,X_\rho u)$ are represented by a $(h\times \rho)$-matrix-valued measure with finite total variation. 
    The first main result of this paper is the proof of the rank-one property for $BV_H$ functions, where we denote by $|D_H^s u|$ the total variation of the singular part $D_H^s u$ of $D_Hu$ with respect to the Haar measure $\mathscr L^d$.
    
    \begin{theorem}[Rank-one Theorem in Carnot groups]
        \label{rank-one thm}
        Let $\Omega$ be an open subset of a Carnot group  $\GG$. Then, for any $u \in BV_{H,\loc}(\Omega;\RR^h)$ the singular part $D_H^su$ of $D_Hu$ is a rank-one measure, i.e.
        \[
            \frac{dD_Hu}{d|D_Hu|}(x) \quad \text{has rank one for } |D_H^su| \text{-almost every } x \in \Omega,
        \]
        where $\frac{dD_Hu}{d|D_Hu|}$ denotes the (matrix-valued) Radon-\Nikodym \ derivative of $D_Hu$ with respect to its total variation $|D_Hu|$.
    \end{theorem}
    
    Theorem \ref{rank-one thm} has been proved in the affirmative only when  $\GG$ belongs to a certain class of Carnot groups that includes the Heisenberg groups $\HH^n$  with $n\geq 2$, see~\cite{don-massaccesi-vittone}. The strategy  in~\cite{don-massaccesi-vittone} follows the geometric one devised in~\cite{massaccesi-vittone} and it does not cover general Carnot groups: the case of the first Heisenberg group $\HH^1$, for instance, was left as an open problem. On the contrary, the proof of Theorem~\ref{rank-one thm} we present here is based on the following sub-Riemannian version of~\cite[Theorem~1.1]{DPR}: this is the second main result of our paper.
            
    \begin{theorem}\label{main theorem}
        Let $\Omega \subseteq \GG$ be an open set and let $\mathscr{A}$ be a left-invariant differential operator of order $k$. Then, for any $\mathscr{A}$-free\footnote{A Radon measure $\mu \in \mathcal{M}(\Omega;\RR^m)$ is \emph{$\mathscr{A}$-free} if $\mathscr{A} \mu = 0$ in the sense of distributions, see Definition~\ref{def_Afree}.} vector-valued measure $\mu$ we have
        \[
            \frac{d\mu}{d|\mu|}(x) \in \Lambda^H_\mathscr{A} \qquad \text{for } |\mu^s|\text{-a.e. } x \in \Omega,
        \]
        where $\mu^s$ is the singular part in the Lebesgue-Radon-\Nikodym \ decomposition of $\mu$ w.r.t. the Haar measure $\Leb[d]$ on $\GG \equiv \RR^d$.
    \end{theorem}

    The symbol $\Lambda^H_\mathscr A$ denotes the {\em hypoelliptic wave cone} of $\mathscr A$, introduced in Definition~\ref{definition hypoelliptic wavecone}, which is in general contained in the (Euclidean) wave cone $\Lambda_\mathscr A$ considered in~\cite{DPR}: see Remark~\ref{rem_wavehypowave}. Theorem~\ref{rank-one thm} follows by applying Theorem~\ref{main theorem} to a suitable curl-type operator $\mathscr A^s$, introduced in Carnot groups of step $s$ in Section~\ref{necessary constraints}, for which  any horizontal gradient $D_Hu$ is always $\mathscr A^s$-free and such that the cone $\Lambda^H_{\mathscr A^s}$ is constituted by matrices of rank-one. It is worth noticing that the  inclusion $\Lambda^H_{\mathscr A} \subsetneq \Lambda_{\mathscr A}$  can be strict (see also Remark~\ref{rem_dueconiHeisenberg} below):  it was observed  in~\cite[Remarks~5.5 and~5.6]{don-massaccesi-vittone}  that the Euclidean wave cone  of the {\em Heisenberg curl} coincides with the full space of matrices\footnote{While the hypoelliptic wave cone  is the set of rank-one matrices.}, thus preventing a direct application of~\cite[Theorem~1.1]{DPR}. In this sense, Theorem~\ref{main theorem} provides a meaningful improvement of~\cite[Theorem~1.1]{DPR}.
    \begin{remark}
        It is known that the vanishing of the Heisenberg curl is a necessary and sufficient condition for a measure to be a horizontal gradient. One might wonder whether, after getting rid of some proper redundancies, our operators $\mathscr A^s$ provide also {\em sufficient} conditions in general Carnot groups, or at least in {\em free} ones (see ~\cite[§9.1.1 and §11.1.3.6]{LeDonne_libro}). We did not investigate this question as it was not among the scopes of the present paper;  let us mention, however, that this is true in $\HH^1 $ by considering certain components of $\mathscr A^2$ -- namely, the operator\footnote{This operator, by the way, coincides with the well-known one introduced by M.~Rumin~\cite{Rumin94JDG}.} $(\mathscr{A}^2_{(1,1,2)},\mathscr{A}^2_{(2,1,2)})$, see Remark~\ref{rem_rumin}.
    \end{remark}

    The proof of Theorem \ref{main theorem} is inspired by the strategy of the original Euclidean proof of~\cite[Theorem~1.1]{DPR}, with some important technical obstructions to be remedied.
    First, \cite[Theorem 1.1]{DPR} relies heavily on the classical Besicovitch differentiation theorem, which may fail in Carnot groups endowed with the Carnot-Carathéodory distance; see e.g.~\cite{LeDonne-Rigot}. We avoid this possibility by proving a weaker localized version of the differentiation theorem in Appendix \ref{app_Besicovitch}, which is enough for the sake of the proof.
    Second, the existence of ``intrinsic'' tangent measures (at almost every point) is fundamental, and Appendix \ref{app_misuretangenti} provides an elementary generalization to Carnot Groups of well-known results in Geometric Measure Theory. In this regard, we refer the curious reader to the recent advances made in \cite{Antonelli-Merlo_existence, Antonelli-Merlo_Marstrand-Mattila, Antonelli-Merlo_representation}.

    The idea of the proof is to argue by contradiction, blowing up $\mu^s$ around the base point at which the assertion fails. A contradiction is reached by showing that we can choose a blow-up sequence (of singular measures) that converges weak-* and in total variation to a non-null tangent measure which is absolutely continuous w.r.t. the Haar measure. To prove this, one mollifies and localizes the blow-up sequence and concludes via a corollary of the Vitali convergence theorem (which is akin to an $L^1$ compensated compactness method). The key point here is that in \cite{DPR} this is achieved via pseudo-differential calculus, and in particular by continuity and compactness results for Singular Integral operators, such as the \Hormander-Mihlin multiplier theorem. A similar theory has been developed, in the non-commutative setting of Nilpotent Lie Groups, starting from the $1970$'s, and the corresponding results are what we employ in the present paper. The caveat is that Fourier Analysis no longer suffices, whereas the right language is that of Representation Theory. Similarly, the role played by ellipticity in the definition of the wave cone $\Lambda_\mathscr{A}$ in the Euclidean setting is replaced by {\em hypoellipticity}, which is ultimately what gives rise to the definition of the hypoelliptic wave cone $\Lambda^H_\mathscr{A}$.
    
    \begin{remark}
        As stated, the proof of Theorem~\ref{main theorem} presented in §\ref{subsec_proof} works for left-invariant operators, but it can be easily adapted for variable coefficients operators, e.g. as in \eqref{eq A}, provided  the coefficients $A_\alpha : \Omega \rightarrow \RR^{n \times m}$ are sufficiently regular. In that case, the statement reads
        \[
        \frac{d\mu}{d|\mu|}(x) \in \Lambda^H_\mathscr{A}(x) \qquad \text{for } |\mu|^s\text{-a.e. } x \in \Omega.
        \]
        Moreover, Theorem \ref{main theorem} holds also for measures $\mu$ which satisfy
        \[
        \mathscr{A} \mu = \sigma \qquad \text{for some } \sigma \in \mathcal{M}(\Omega; \RR^n)
        \]
        and not just in the particular case in which $\sigma = 0$.
        See \cite[Remarks 1.3 and 1.4]{DPR} for more details.
    \end{remark}\medskip

    {\bf Plan of the paper.} 
    The paper is organized as follows. In Section~\ref{sec_Carnot} we introduce Carnot groups and the relevant notation, in particular about Heisenberg groups. In Section~\ref{sec_reprhypo} we introduce some classical material about Lie group representations, hypoelliptic operators and singular integrals that are needed for the proof of Theorem~\ref{main theorem}, which is provided in Section~\ref{sec_mainresult}. In Section~\ref{sec_rk1} we prove Theorem~\ref{rank-one thm} by introducing  curl-type operators $\mathscr A^s$ and then computing their hypoelliptic wave cone. Eventually, the Appendices contain the proof of the Besicovitch-type differentiation theorem, Theorem~\ref{weak besicovitch thm Carnot}, and of the existence of tangent (in an intrinsic sense) measures, Proposition~\ref{prop existence of tangent measures}.\medskip
    
    {\bf Acknowledgments.}
    The authors wish to thank Gian Maria Dall'Ara for the fruitful discussions and directions regarding the Harmonic Analysis literature on Homogeneous Spaces. \\
    The authors have been supported by the University of Padova and they are members of the Istituto Nazionale di Alta Matematica (INdAM), Gruppo Nazionale per l'Analisi Matematica, la Probabilità e le loro Applicazioni (GNAMPA). 
    G.D.P.’s research is funded by the European Research Council (ERC) through CoG 101169953 ``RISE''\footnote{Views and opinions expressed are however those of the authors only and do not necessarily reflect those of the European Union or the European Research Council.}. 
    A.G.  acknowledges  support from the Scuola Normale Superiore as well as from the Scuola Galileiana di Studi Superiori, where part of this work was carried out during his time as a student. 
    A.M. and D.V. received funding through the INdAM project {\em VAC\&GMT}. 
    D.V. was also supported by the INdAM-GNAMPA 2026 Project \emph{Variational, Geometric, and Analytic Perspectives on Regularity}, CUP E53C25002010001.\medskip

     {\bf Statement of AI use.}
    The authors used AI tools for minor copyediting assistance during the final revision of this paper. 
    The manuscript and all mathematical content were conceived and written entirely by the authors.

	\section{Preliminaries in Carnot groups}\label{sec_Carnot}
	\subsection{Carnot groups}\label{subsec_Carnot}
	A Carnot (or {\em stratified})  group $\GG$ (see e.g.~\cite{LeDonne_libro}) is a connected, simply connected and nilpotent Lie group whose Lie algebra $\g$ is {\em stratified}, i.e., it possesses a decomposition  $\g=\g_1\oplus\dots\oplus\g_s$ such that
	\[
	\forall\ j=1,\dots,s-1\quad \g_{j+1}=[\g_j,\g_1],\qquad\g_s\neq\{0\}\qquad\text{and}\qquad [\g_s,\g]=\{0\}.
	\]
	We refer to the integer $s$ as the {\em step} of $\GG$ and to $\rho:=\,$dim $\g_1$ as its {\em rank}; we also denote by $d$ the topological dimension of $\GG$. The group identity is denoted by $0$ and, as customary, we identify $\g$, $T_0\GG$ and the algebra of left-invariant vector fields on $\GG$. The elements of $\g_1$ are referred to as {\em horizontal}.
	
	The exponential map $\exp:\g\to\GG$ is a diffeomorphism and, given a basis $X_1,\dots,X_d$ of $\g$, we will often identify $\GG$ with $\RR^d$ by means of exponential coordinates:
	\[
	\RR^d\ni x=(x_1,\dots,x_d)\longleftrightarrow \exp\left( x_1X_1+\dots+x_dX_d\right)\in\GG.
	\]
	We will say that a basis is \emph{adapted} to the stratification if
	\begin{align*}
		&  X_1,\dots,X_\rho \text{ is a basis of $\g_1$, and}\\
		& \forall\ j=2,\dots,s,\ X_{\dim (\g_1\oplus\cdots\oplus\g_{j-1})+1},\dots, X_{\dim (\g_1\oplus\cdots\oplus\g_{j})}\text{ is a basis of $\g_j$.} 
	\end{align*}
	A one-parameter family $\{\delta_\lambda\}_{\lambda>0}$ of {\em dilations} $\delta_\lambda:\g\to\g$  is defined by (linearly extending)  
	\[
	\delta_\lambda(X):=\lambda^j X \text{ for any } X\in\g_j;
	\]
	notice that dilations are Lie algebra homomorphisms and $\delta_{\lambda\mu}=\delta_\lambda\circ\delta_\mu$. By composition with $\exp$ one can then define a one-parameter family, for which we use the same symbol, of group isomorphisms $\delta_\lambda:\GG\to\GG$.
	
	We fix a left-invariant homogeneous  distance $d$ on $\GG$ (for instance the \emph{Carnot-Carathéodory} distance, see e.g.~\cite{LeDonne_libro}), i.e., a distance such that
	\[
	d(xy,xz)=d(y,z), \quad d(\delta_\lambda x,\delta_\lambda y)=\lambda d(x,y)\qquad\text{for all }x,y,z\in\GG,\lambda>0.
	\]
	We use $d$ to denote both the distance on $\GG$ and its topological dimension, but no confusion  will ever arise. We denote by $B^\GG(x,r)$ the open ball of center $x\in\GG$ and radius $r>0$ and by $B^{\RR^d}(x,r)$ the open Euclidean ball of center $x$ and radius $r$. We have the following well-known fact (see e.g.~\cite[Corollary~4.3.4]{LeDonne_libro}).
	\begin{corollary}
		\label{estimate d_c}
		For every compact set $K\subset\GG$ there exists a constant $C=C(K)>0$ such that, for every $x\in K$ and every $r\in(0,1)$, the inclusion $B^{\RR^d}(x,r^s)\subset B^\GG(x,C r)$ holds.
	\end{corollary}
    Recall that $\Leb[d]$ is a Haar measure on $\GG\equiv\RR^d$ and that the homogeneous dimension of $\GG$ is the integer $Q:=\sum_{j=1}^s j\dim\g_j$. One has
    \[
    \Leb[d] (B^\GG(x,r)) = r^Q \Leb[d] (B^\GG(0,1))\qquad\text{for all }x\in\GG,r>0.
    \]
    The number $Q$ is always greater than $d$ (apart from the Euclidean case $s=1$) and it coincides with the Hausdorff dimension of $\GG$. Since also the Hausdorff $Q$-dimensional measure is a Haar measure, it coincides with $\Leb[d]$ up to a multiplicative constant.  
    
    Following \cite[Section 1.B]{Folland-Stein} we also recall the definition of convolution on groups: given two measurable functions $f,g$ on $\GG$, we define
    \[
    f*g(x) := \int_\GG f(xy^{-1})g(y) \ d\Leb[d](y) = \int_\GG f(y)g(y^{-1}x) \ d\Leb[d](y)
    \]
    whenever the integrals converge. We have the following well-known fact: take any $X \in \g$ left-invariant vector field on $\GG$ and define $X^R$ its associated \emph{right-invariant} vector field, defined via $X^R\vert_0 = X\vert_0$. Then
    \[
    X(f*g) = f*(Xg), \quad X^R(f*g) = (X^Rf)*g, \quad (Xf)*g = f*(X^Rg).
    \]
    
    
    
    \subsection{Heisenberg groups}
    \label{subsec_heisenberg}
    For every $n\geq 1$, the $n$-th Heisenberg group $\HH^n$ is the Carnot group of dimension $d=2n+1$, step $s=2$ and rank $\rho=2n$ associated with the $(2n+1)$-dimensional stratified Lie algebra $\h$ generated by elements $X_1,\dots,X_n,Y_1,\dots,Y_n,T$ whose Lie brackets vanish except for 
    \[
    [X_j,Y_j]=T\qquad\text{for every }j=1,\dots,n.
    \]
    The algebra stratification is given by $\h=\h_1\oplus\h_2$, where
    \[
    \text{$\h_1:=\Span\{X_1,\dots,X_n,Y_1,\dots,Y_n\},\qquad \h_2:=\Span\{T\}$;}
    \]
    it follows that the second layer $\h_2$ is the center of the algebra.
    
    We identify $\HH^n$ with $\RR^{2n+1}=\RR^n\times\RR^n\times\RR$ by means of exponential coordinates
    \[
    \HH^n\ni\exp(x_1X_1+\dots+x_nX_n+y_1Y_1+\dots+y_nY_n+tT)\longleftrightarrow(x,y,t)\in\RR^{2n+1}.
    \]
In exponential coordinates, the group law reads as
    \[
    (x,y,t)(x',y',t')=(x+x',y+y',t+t'+ \tfrac12 \langle x,y'\rangle-\tfrac12\langle x',y\rangle)
    \]
    and left-invariant vector fields have the form
    \[
    X_j=\partial_{x_j}-\frac{y_j}2 \partial_t,\qquad Y_j=\partial_{y_j}+\frac{x_j}2\partial_t,\qquad T=\partial_t.
    \]
    In this case, a one-parameter family $(\delta_\lambda)_{\lambda>0}$ of group automorphisms is provided by the dilations $\delta_\lambda(x,y,t):=(\lambda x,\lambda y,\lambda^2 t)$.

    As customary, when $n=1$ we write for simplicity $x,y,X,Y$ in place of, respectively, $x_1,y_1,X_1,Y_1$. 

    \section{Representations and hypoellipticity}\label{sec_reprhypo}
    
    In this section we introduce a few standard definitions and results in Representation Theory in order to define Rockland operators through Theorem \ref{thm_rockland}, which characterizes hypoellipticity in graded Lie groups (see e.g. \cite[§3.1]{Fischer-Ruzhansky}), and consequently present some properties of Bessel operators. For more detail, we refer to \cite{Fischer-Ruzhansky}.

    \subsection{Lie group representations and Rockland operators}
    
    \begin{definition} [Lie group representations]
        A \emph{representation} of a Lie Group $\GG$ is an action of $\GG$ on a vector space $V$, i.e. a smooth group homomorphism $\Pi : \GG \rightarrow GL(V)$ where $GL(V)$ is the general linear group on $V$. We also say that $\Pi$ is \emph{irreducible} whenever its only invariant subspaces are $V$ and the zero space, where $W \leq V$ is said to be invariant if $\Pi(g)(W) \subseteq W$ for any $g \in \GG$. Moreover, we say that $\Pi$ is a \emph{unitary representation} if $V$ is an Hilbert space and $\Pi(\GG)$ is mapped into the space of unitary operators\footnote{Recall that an unitary operator on an Hilbert space $H$ is a bounded linear operator $U : H \rightarrow H$ such that $U^*U = UU^* = \operatorname{Id}_H$.} on $V$. \\
        We say that two representations $\Pi_1, \Pi_2$ on (respectively) $V_1, V_2$  are equivalent if there exists a bounded invertible linear map $A: V_1 \rightarrow V_2$ such that $A \Pi_1 = \Pi_2 A$. We will denote by $\widehat{\GG}$ the ``dual'' space of equivalence classes of irreducible unitary representations of $\GG$.
    \end{definition}
    
    \begin{example} [\Schrodinger \ representation] \label{ex_schrodinger}
        Take a non-zero real number $\hbar$ and an Heisenberg group $\HH^n$ as in §\ref{subsec_heisenberg}. Then we may define an irreducible unitary representation $\Pi_\hbar$ of $\HH^n$, known as the \emph{\Schrodinger \ representation of $\HH^n$}, acting on the Hilbert space $L^2(\RR^n;\CC)$ via
        \begin{equation}
            \label{schrodinger representation}
            \big[ \Pi_\hbar(x,y,t) \psi \big](z) := e^{i\hbar(y \cdot z + t + \frac{x \cdot y}{2})} \psi(z+x)
        \end{equation}
        for all $(x,y,t) \in \RR^n \times \RR^n \times \RR \equiv \ \HH^n$ (exponential coordinates are used here), $\psi \in L^2(\RR^n; \CC)$ and $z \in \RR^n$.
    \end{example}
    
    Following \cite[§2]{Rockland}, any given unitary representation $\Pi$ of a Lie group $\GG$ onto some Hilbert space $H$ induces a Lie algebra representation\footnote{With a slight abuse of notation we use the same symbol for both.} $\Pi$ of $\g$ onto the vector subspace $H_\infty \subseteq H$ of $C^\infty$-vectors, i.e. the space of vectors $v \in H$ such that the map $x \mapsto \Pi(x)v$ from $\GG$ to $H$ is smooth. More precisely, we define
    \[
          \Pi(Z)v := \left.\frac{d}{dt}\right\vert_{t=0} \Pi(\exp tZ)v \qquad \text{for } Z \in \g, v \in H_\infty,
    \]
    which gives the representation $\Pi : \g \rightarrow L(H_\infty)$. This extends uniquely to a representation of the algebra of left-invariant differential operators on $\GG$, since any element $P$ of said algebra can be uniquely written as
    \[
          P = \sum_{|\alpha|\leq m} b_\alpha X_1^{\alpha_1} \cdots X_n^{\alpha_n},
    \]
    where $X_1, \ldots, X_n$ form an ordered basis for $\g$.
    
    If $\Pi$ is irreducible, it may be shown that there is a unitary equivalence taking $H$ to $L^2(\RR^n;\CC)$ for some $n \in \NN$ and taking $H_\infty$ to  the space $\mathscr{S}(\RR^n;\CC)$ of Schwartz functions on $\RR^n$.

    \begin{example}
        Following Example \ref{ex_schrodinger} above, one can easily check that the standard vector fields $X_1, \ldots, X_n, Y_1, \ldots, Y_n, T$ introduced in \ref{subsec_heisenberg}, act as follows when represented via \Schrodinger:
        \[
            \Pi_h [X_j] \psi = \frac{\partial \psi}{\partial z_j}, \quad \Pi_h [Y_j] \psi = ihz_j \cdot \psi, \quad \Pi_h [T] \psi = ih \cdot \psi,
        \]
        for $\psi \in \mathscr{S}(\RR^n;\CC)$. Note that here $i=\sqrt{-1}$.
    \end{example}
    
    In order to give some context to the following discussion, we state this result due to \Hormander:
    
    \begin{theorem} (\cite[Theorem 19.5.1]{HormanderIII})
        \label{thm Hormander}
        Let $D$ be a linear differential operator of order $k$ defined on the space of smooth functions on a smooth manifold $M$. Then the following are equivalent
        \begin{enumerate}[label=(\alph*)]
            \item The operator $D$ is elliptic, i.e. for any $\xi \in T^*M\backslash\{0\}$ the principal symbol\footnote{The principal symbol is the symbol of the $k$-homogeneous part of $D$} of $D$ at $\xi$ never vanishes.
            \item For every (or for some) $s \in \RR$ and for every distribution $u$ on $M$ we have the following: if the distributional derivative $Du$ belongs to the the local $s$-Sobolev $L^2$ space $H^s(M) = W^{s,2}(M)$, then $u \in H^{s+k}(M)$.
        \end{enumerate}
    \end{theorem}
    
    We say that an operator is \emph{hypoelliptic} whenever solving for $u$ yields regularity, similar to $(b)$ above:
    \begin{definition}[Hypoellipticity]
        A differential operator $P$ on a manifold $M$ is hypoelliptic if, for any distribution $u$ on an open subset $\Omega \subseteq M$, $Pu \in C^\infty(\Omega)$ implies that $u \in C^\infty(\Omega)$.
    \end{definition}
    
    In the celebrated paper~\cite{Hormander67}, \Hormander \ showed the existence of hypoelliptic operators which are not elliptic:
    
    \begin{theorem}\label{thm_Hormanderdacitare} (\Hormander, \cite{Hormander67})
        Let $X_0,X_1, \ldots, X_{\rho}$ be a family of vector fields on $M$ satisfying \Hormander's condition (i.e. which bracket-generate $T_xM$ at any point $x \in M$). Then $D = X_0+\sum_{i=1}^\rho X_i^2 $ is hypoelliptic.
    \end{theorem}
    
    We shall be concerned with the hypoellipticity of left-invariant operators on graded Lie groups (in particular, Carnot Groups). A full characterization was given by Rockland \cite[Theorem 1.2]{Rockland} for Heisenberg groups, later generalized by Helffer and Nourrigat \cite[Théorème 0.1]{Helffer-Nourrigat}:
    \begin{theorem}
        \label{thm_rockland}
        Let $P$ be a left-invariant homogeneous differential operator on a graded nilpotent, connected and simply connected Lie group $\GG$. Then the following are equivalent:
        \begin{enumerate}[label=(\alph*)]
            \item $P$ is hypoelliptic
            \item For every irreducible unitary representation\footnote{Except the trivial identity representation, i.e. the degenerate representation that represents all elements of $\g$ as the zero operator on $C^\infty$-vectors.} $\Pi \in \widehat{\GG}$, $\Pi(P)$ is injective, i.e. $\Pi(P)v \not= 0$ for any non-zero $C^\infty$-vector $v$ of the representation $\Pi$.
        \end{enumerate}
    \end{theorem}
    
    These results also led to the definition of \emph{Rockland operators}, i.e. operators $P$ as above which satisfy condition $(b)$. Following \cite[Chapter 4]{Fischer-Ruzhansky}, we note that Miller showed that the existence of a positive Rockland operator is equivalent to $\GG$ being a graded Lie group (among all Lie groups endowed with a homogeneous structure). Through this positive Rockland operator (e.g. the sublaplacian in Heisenberg groups), much work has been done to develop the theory of Riesz/Bessel operators and pseudodifferential calculus on graded Lie groups. The following are some of these properties which will be useful during the proof of Theorem \ref{main theorem}.

    \subsection{Bessel potentials and Singular Integrals}
    
    In the rest of the paper, we will refer to a single ``classical'' positive Rockland operator for each group (see \cite[Lemmas 4.1.8, 4.2.3 and Corollary 4.1.10]{Fischer-Ruzhansky}), which in the stratified case (such as Carnot groups) is the standard sublaplacian operator. Nonetheless, the following results hold for any such choice.
    
    \begin{definition}(Bessel potential)
        \label{definition bessel potentials}
        Let $\GG$ be a Carnot group and $\mathcal{R}$ its positive Rockland operator, and call $\theta$ the homogeneous degree of $\mathcal{R}$. Then define for $\alpha > 0$ the \emph{Bessel potential} operator $(1 + \mathcal{R})^{-\alpha/\theta}$ as the fractional power of $\operatorname{Id} + \mathcal{R}$ (see for example \cite[§A.3]{Fischer-Ruzhansky}).
    \end{definition}
    
    As in the Euclidean space, through Bessel (or Riesz) potentials one can define appropriate notions of ``anisotropic'' Sobolev spaces (see e.g. \cite[§4.4]{Fischer-Ruzhansky}). In that light, the following can be viewed as a Sobolev embedding compactness result:
    
    \begin{lemma}
        \label{lemma bessel potentials}
        Given $\alpha>0$, the operator $(1 + \mathcal{R})^{-\alpha/\theta}$ can be extended to a bounded right-convolution operator from $L^p(\GG)$ ($1 \leq p < \infty$) to itself. Its convolution kernel $\mathcal{J}_\alpha$ is smooth on $\GG \backslash \{0\}$ and is in $L^p(\GG)$ for any $1 \leq p < p(\alpha,Q) := \min\left\{2, {Q}/{(Q-\alpha)}\right\}$,  $Q$ being the homogeneous dimension of $\GG$. \\
        Moreover,  $(1 + \mathcal{R})^{-\alpha/\theta}$ extends to a compact operator from $L^1_c(B^\GG_1)$ to $L^p(\GG)$ for any $1 \leq p < p(\alpha,Q)$.
    \end{lemma}
    \begin{proof}
        The first part of the statement is taken from \cite[Corollary 4.3.11]{Fischer-Ruzhansky}, and $\mathcal{J}_\alpha \in L^p(\GG)$ for $p < p(\alpha,Q)$ as a consequence of \cite[Corollary 4.3.13]{Fischer-Ruzhansky}. The second part follows as in the proof of \cite[Lemma 2.1]{DPR}: by density of $C^1_c(\GG)$ in $L^1(\GG)$, we can write $(1 + \mathcal{R})^{-\alpha/\theta}$ as a limit (in operator norm) of compact right-convolution operators whose kernels are in $C^1_c(\GG)$.
    \end{proof}
    
    Recall that a fundamental solution $K$ of a differential operator $\mathcal{L}$ is a distribution such that $\mathcal{L}K = \delta_0$, the Dirac delta at zero. In particular, with fundamental solutions we may invert operators through convolution:
    \[
    \mathcal{L}(f*K) = f*\mathcal{L}K = f
    \]
    
    \begin{lemma} (\cite[Theorem 3.2.40]{Fischer-Ruzhansky})
        \label{lemma fundamental solution}
        Let $\mathcal{L}$ be a hypoelliptic, left-invariant homogeneous differential operator on $\GG$ such that its formal adjoint $\mathcal{L}^*$ is also hypoelliptic. Then, $\mathcal{L}$ admits a fundamental solution $K \in \mathscr{S}'(\GG)$ which is smooth on $\GG \backslash \{0\}$.
    \end{lemma}
    
    We denote by $L^{1,\infty}(X,\mu)$ the space of $\mu$-measurable functions $f$ such that the quasi-norm $\norm[f][L^{1,\infty}] := \sup \{ \alpha \mu(\{x \in X: |f(x)| > \alpha\}) : \alpha > 0 \}$ is finite. We will also call \emph{of strong-type} $(p,q)$ any operator which extends continuously from $L^p$ to $L^q$. When it extends continuously from $L^p$ to $L^{1,\infty}$ it is instead called \emph{of weak-type} $(p,1)$. The following lemma can be viewed as an analog to the \Hormander-Mihlin multiplier theorem:
    
    \begin{lemma}
        \label{lemma singular integrals}
        Let $\mathcal{L}$ be as in Lemma \ref{lemma fundamental solution}, and call $K$ its fundamental solution. Then, for any left-invariant differential operator $P$ which is homogeneous of the same order as $\mathcal{L}$, the operator
        \[
        T:\mathscr{D}(\GG) \rightarrow \mathscr{D}'(\GG), \qquad Tu := Pu*K
        \]
        can be extended continuously from $L^p(\GG)$ to itself for any $1 < p < \infty$ and from $L^1(\GG)$ to $L^{1,\infty}(\GG)$.
    \end{lemma}
    \begin{remark}
        Consider, with the notation above, the equation
        \[
        \mathcal{L}u = Pv.
        \]
        Then the lemma above is implying that the map
        \[
        v \longmapsto u
        \]
        is of strong-type $(p,p)$ for $1 < p < \infty$ and of weak-type $(1,1)$.
    \end{remark}
    \begin{proof}[Sketch of proof]		
        The statement follows from \cite[§XIII.3.1]{stein-murphy_harmonic_analysis}, in particular Remark (ii), when $\GG$ is an Heisenberg Group $\HH^n$ (see §\ref{subsec_heisenberg}), but the same argument can be repeated in any stratified Lie group (similar arguments are used during the proof of \cite[Theorem 2]{Nagel-Ricci-Stein}, but in much more generality).
        The weak $(1,1)$ bound for $T$ is not explicitly stated but holds nonetheless: see \cite[Theorem A.4.4]{Fischer-Ruzhansky}: essentially, the authors in \cite{stein-murphy_harmonic_analysis} and \cite{Nagel-Ricci-Stein} prove that $T$ is a singular integral operator in the sense of Coifman-Weiss (see \cite[Chapitre III]{Coifman-Weiss}).
    \end{proof}
    
    \section{Structure of \texorpdfstring{$\mathscr{A}$}{A}-free measures in Carnot Groups}\label{sec_mainresult}
In this Section we prove Theorem~\ref{main theorem}; before that, however, we need to fix some notation and state some preliminary results.
    
    Consider a Carnot Group $\GG \equiv \RR^d$ of step $s$ and rank $\rho$ and denote by $X = (X_1, \ldots, X_\rho)$ a basis for its horizontal bundle $\g_1$. Consider an open subset $\Omega \subseteq \GG$ and a left-invariant linear differential operator $\mathscr{A}$ of order $k$, i.e.
    \begin{gather}
        \notag \mathscr{A} : \mathscr{D}'(\Omega;\RR^m) \rightarrow  \mathscr{D}'(\Omega;\RR^n) \\
        \label{eq A} [\mathscr{A} u](x) = \sum_{|\alpha| \leq k} A_\alpha(x) X^\alpha u \qquad \text{for } x \in \Omega,
    \end{gather}
    where $A_\alpha(x) \in \RR^{n \times m}$ and, for each $\alpha = (\alpha_1, \ldots, \alpha_h) \in \{1,\dots,\rho\}^h$, we set $|\alpha|:=h$ and $X^\alpha = X_{\alpha_1} \cdots X_{\alpha_h}$.

\begin{definition}\label{def_Afree}
We say that a Radon measure $\mu \in \mathcal{M}(\Omega;\RR^m)$ is \emph{$\mathscr{A}$-free} whenever
    \[
    \mathscr{A} \mu = 0 \quad \text{in } \mathscr{D}'(\Omega;\RR^n),
    \]
    where $\mu$ is seen as a distribution in $\mathscr{D}'(\Omega;\RR^m)$ of order $0$.
\end{definition}

    Now let $\mathscr{A}^k$ be the $k$-homogeneous part of $\mathscr{A}$. For any $\lambda \in \RR^m$, define the following operator
    \[
    \mathcal{B}_\lambda : \mathscr{D}'(\Omega) \rightarrow  \mathscr{D}'(\Omega;\RR^n), \qquad \mathcal{B}_\lambda \varphi := \mathscr{A}^k(\lambda \varphi).
    \]
    and denote by $B^*_\lambda$ its formal transpose, that is
    \[
    \langle \mathcal{B}_\lambda^*(\xi), \varphi \rangle = \langle \xi, \mathcal{B}_\lambda(\varphi) \rangle
    \]
    for any $\varphi \in \mathscr{D}(\Omega), \xi \in \mathscr{D}(\Omega;\RR^n)$.
    
    \begin{definition}[Hypoelliptic wave cone]
        \label{definition hypoelliptic wavecone}
        With the notation above, we define the \emph{hypoelliptic wave cone} $\Lambda^H_\mathscr{A}(x) \subseteq \RR^m$ of $\mathscr{A}$ at $x \in \Omega$ via
        \[
        \lambda \not\in \Lambda^H_\mathscr{A}(x) \quad \Longleftrightarrow \quad \mathcal{B}^*_\lambda \mathcal{B}_\lambda \text{ is hypoelliptic near } x.
        \]
        Notice that whenever $\mathscr{A}$ is left-invariant, the dependence on $x$ can be dropped. In that case, we will just write $\Lambda^H_\mathscr{A}$.
    \end{definition}
    The following equivalent characterization is an immediate corollary of Theorem \ref{thm_rockland}:
    \begin{lemma}
        \label{lemma equivalence hypoelliptic wave cone}
        The hypoelliptic wave cone $\Lambda^H_\mathscr{A}(x)$ coincides with the space of $\lambda$'s such that the linear map
        \[
        \Pi (\mathcal{B}^*_\lambda\mathcal{B}_\lambda)(x) : H_\infty \rightarrow H_\infty
        \]
        is not injective for some $\Pi \in \widehat{\GG}$, where $H_\infty$ is the space of $C^\infty$-vectors of $\Pi$.
        That is,
        \[
        \Lambda^H_\mathscr{A}(x) = \left\{ \lambda \in \RR^m : \exists \Pi \in \widehat{\GG} \text{ such that } \ker \Pi (\mathcal{B}^*_\lambda \mathcal{B}_\lambda)(x) \not= 0 \right\}.
        \]
    \end{lemma}
    
    \begin{remark}\label{rem_wavehypowave}
        Recall that the \emph{wave cone} $\Lambda_\mathscr{A}(x)$ of $\mathscr{A}$ is defined as (see~\cite{DPR})
        \[
        \lambda \not\in \Lambda_\mathscr{A}(x) \quad \Longleftrightarrow \quad \mathcal{B}_\lambda \text{ is elliptic near } x.
        \]
        Clearly, $\Lambda_\mathscr{A} \supseteq \Lambda^H_\mathscr{A}$. Therefore, Theorem \ref{main theorem} provides a net improvement of \cite[Theorem 1.1]{DPR} whenever $\Lambda_\mathscr{A} \supsetneq \Lambda^H_\mathscr{A}$, see for instance Remark~\ref{rem_dueconiHeisenberg} below.
    \end{remark}

    \begin{remark}\label{rem_partivarie}
        As it is implicitly used throughout the paper, it is worth observing that the total variation $|\mu^s|$ of the singular part $\mu^s$ coincides with the singular part $|\mu|^s$ of the total variation $|\mu|$. A similar consideration applies to the absolutely continuous parts, for which one has the equality $|\mu^a|=|\mu|^a$.
        In particular, one has $\frac{d\mu}{d|\mu|}=\frac{d\mu^s}{d|\mu^s|}$ $|\mu^s|$-almost everywhere.
    \end{remark}

    \subsection{Tools}\label{subsec_tools}
    We state here a few technical results that will be used in the proof of Theorem~\ref{main theorem}. The first one  is taken directly\footnote{Apart from the obvious observation that $B^\GG_1$ is contained in some Euclidean ball.} from \cite[Lemma 2.2]{DPR}; we denote by $L^1_c(B^\GG_1)$  the space of $L^1$ functions supported on $B^\GG_1$.
    \begin{lemma}
        Let $\{f_j\}_j \subseteq L^1_c(B^\GG_1)$ be a sequence such that
        \begin{enumerate}[label=(\alph*)]
            \label{corollary Vitali conv thm}
            \item $f_j \overset{*}{\rightharpoonup} 0$ in $\mathscr{D}'(B^\GG_1)$;
            \item the negative parts $f_j^-$ vanish in measure, i.e.
            \[
            \lim_{j \rightarrow \infty} \Leb[d](\{f_j^- > t\}) = 0 \quad \text{for any } t > 0;
            \]
            \item the negative parts $f_j^-$ are equi-integrable, i.e.
            \[
            \lim_{\Leb[d](E) \rightarrow 0} \sup_{j \in \NN} \int_{E} f_j^- \ d\Leb[d] = 0.
            \]
        \end{enumerate}
        Then, $f_j \rightarrow 0$ in $L^1_\loc(B^\GG_1)$.
    \end{lemma}
    
    The second tool is a “weak" Besicovitch-type differentiation theorem in Carnot groups, which differs  from the classical statement in that the limit is taken only along a {\em sequence} of infinitesimal radii. 

    \begin{theorem}
        \label{weak besicovitch thm Carnot}
        Let $\nu$ be a Radon measure on $\GG$ and $f \in L^1(\GG, \nu)$. Then we may find, for $\nu$-almost any $x \in \GG$, an infinitesimal sequence $(r_k^x)_k$  such that
        \begin{equation}
            \label{measure differentiation Carnot}
            \lim_{k \rightarrow \infty} \fint_{B^\GG(x,r_k^x)} |f-f(x)| \ d\nu = 0.
        \end{equation}
    \end{theorem}
    
    Theorem~\ref{weak besicovitch thm Carnot} is proved in Appendix~\ref{app_Besicovitch}.
    
    Eventually, the proof of Theorem~\ref{main theorem} uses the language (and the existence) of {\em tangent measures}. 
    Let $\Omega \subseteq \GG$ be an open set and let $\mu \in \mathcal{M}_\loc(\Omega;V)$ be a Radon measure taking values in a real finite dimensional vector space $V$. Given $x_0 \in \supp \mu$ and an infinitesimal sequence of radii $r_j \downarrow 0$, we will denote by
    \[
    \nu_j := \frac{T^{x_0,r_j}_\#\mu}{|\mu|(B^\GG(x_0,r_j))}
    \]
    the blow-up sequence of $\mu$ at $x_0$ along $(r_j)_j$, where
    \[
    T^{x_0,r_j}(x) := \delta_{1/r_j} (x_0^{-1} \cdot x), \qquad \text{i.e.,} \qquad T^{x_0,r_j}_\#\mu(E) = \mu(x_0 \delta_{r_j}(E)).
    \]
    
    \begin{definition}[Tangent measure]
        The space $\operatorname{Tan}(\mu,x_0)$ of {\em tangent measures} to $\mu$ at $x_0$ consists of all weak-* limits of blow-up sequences $\nu_j$ for any choice of radii $r_j \downarrow 0$.
    \end{definition}
    
    The existence of tangent measures
    is stated in the following Proposition, whose proof is postponed to Appendix~\ref{app_misuretangenti}.
    
    \begin{proposition}
        \label{prop existence of tangent measures}
        For any measure $\mu \in \mathcal{M}_\loc(\Omega; V)$, at $|\mu|$-almost every $x_0 \in \supp \mu$ the set $\operatorname{Tan}(\mu,x_0)$ contains a non-zero measure whose support contains the origin.
    \end{proposition}

    \subsection{Proof of Theorem~\ref{main theorem}}\label{subsec_proof}
    We can now prove our main result. The proof largely follows the strategy and the notation of \cite{DPR}. Nonetheless we present it in full detail for the convenience of the reader.
    
    \begin{proof}[Proof of Theorem~\ref{main theorem}]
        We argue by contradiction, assuming that the set
        \[
        E := \left\{ x \in \Omega : \frac{d\mu}{d|\mu|}(x) \not\in \Lambda_{\mathscr{A}}^H \right\},
        \]
        is not $|\mu|^s$-negligible\footnote{The existence of $\frac{d\mu}{d|\mu|}(x)$ is part of the definition of $E$.}, or $|\mu|^s(E) > 0$. Then, thanks to Theorem \ref{weak besicovitch thm Carnot} and Proposition \ref{prop existence of tangent measures} we may find a point $x_0 \in E$ and a sequence $r_j = r_j^{x_0} \downarrow 0$ such that:
        \begin{enumerate}[label={(\textit{\roman*})}]
            \item $\displaystyle \lim_{j \rightarrow \infty} \frac{|\mu|^a(B^\GG_{r_j}(x_0))}{|\mu|^s(B^\GG_{r_j}(x_0))} = 0$ \ and
            \[\lim_{j \rightarrow \infty} \fint_{B^\GG_{r_j}(x_0)} \left| \frac{d\mu}{d|\mu|}(x) - \frac{d\mu}{d|\mu|}(x_0) \right| \ d|\mu|^s(x) = 0;
            \]
            
            \item there exists a non-zero $\nu \in \text{Tan}(|\mu|^s, x_0)$ on $B^\GG_{1/2}$, of the form
            \[
            \nu_j := \frac{T^{x_0,r_j}_\# |\mu|^s}{|\mu|^s(B^\GG_{r_j}(x_0))} \stackrel{*}{\rightharpoonup} \nu;
            \]
            
            \item $\displaystyle \lambda = \frac{d\mu}{d|\mu|}(x_0) \not\in \Lambda^H_{\mathscr{A}}$. In particular, using the notation above, the operator $\mathcal{B}^*_\lambda \mathcal{B}_\lambda$ is hypoelliptic.
        \end{enumerate}
        
        The contradiction arises from the following claims:
        \begin{gather}
            \label{claim 1 main theorem}
            0 \not= \nu \llcorner B^\GG_{1/2} \ll \Leb[d],\\
            \label{claim 2 main theorem}
            \lim_{j \rightarrow \infty} |\nu_j - \nu|(B^\GG_{1/2}) = 0.
        \end{gather}
        Indeed, each $\nu_j$ is singular w.r.t. $\Leb[d]$, hence there exist Borel sets $E_j \subseteq B^\GG_{1/2}$ such that $\Leb[d](E_j) = 0 = \nu(E_j)$ and $\nu_j(E_j) = \nu_j(B^\GG_{1/2})$. Thus by (\ref{claim 2 main theorem})
        \[
        \nu_j(B^\GG_{1/2}) = \nu_j(E_j) \leq |\nu_j - \nu|(B^\GG_{1/2}) + \nu(E_j) = |\nu_j - \nu|(B^\GG_{1/2}) \rightarrow 0,
        \]
        hence $\nu(B^\GG_{1/2}) = 0$, which clearly contradicts (\ref{claim 1 main theorem}). Thus $|\mu|^s(E) = 0$ and the assertion is proven. \\
        
        \emph{Step 1: Mollification and localization}. Assume $x_0 = 0$ without loss of generality, and let $T^r := T^{x_0,r}$. Denoting by $\mathscr{A}^h$ the $h$-homogeneous part of $\mathscr{A}$, for any $\varphi \in \mathscr{D}(\Omega; \RR^n)$ we have
        \begin{align*}
            \langle \mathscr{A}^h \left( T^r_\# \mu \right) , \varphi \rangle &= \int_{\RR^d} \sum_{|\alpha| = h} A_\alpha^t (X^\alpha)^* \varphi(x) \ d\left(T^r_\# \mu\right)(x) \\
            &= \int_{\RR^d} \sum_{|\alpha| = h} A_\alpha^t (X^\alpha)^* \varphi(\delta_{r^{-1}}y) \ d\mu(y) \\
            &= \int_{\RR^d} r^h \sum_{|\alpha| = h} A_\alpha^t (X^\alpha)^* \varphi_r(y) \ d\mu(y) \\
            &= \langle \mathscr{A}^h (r^h \mu) , \varphi_r \rangle,
        \end{align*}
        where $\varphi_r(x) := \varphi(\delta_{1/r}x)$. Hence, summing over $h$,
        \begin{align*}
            \mathscr{A}^k \left( T^r_\# \mu \right) + \sum_{h=1}^{k-1} \mathscr{A}^h \left( r^{k-h} T^r_\# \mu \right) =  \sum_{h=1}^k \mathscr{A}^h \mu = 0.
        \end{align*}
        Calling $c_j = |\mu|^s(B^\GG_{r_j})^{-1}$, we may write
        \begin{equation}
            \label{first calculation main theorem}
            \mathscr{A}^k (\lambda \nu_j) = \mathscr{A}^k (\lambda \nu_j - c_j T^{r_j}_\# \mu) - \sum_{h=1}^{k-1} \mathscr{A}^h \left( r_j^{k-h} c_j T^{r_j}_\# \mu \right).
        \end{equation}
        
        Now consider an approximation of the identity $\varphi_\varepsilon$ as $\varepsilon > 0$ varies. By weak-* lower semicontinuity of the total variation
        \[
        |\nu_j - \nu|(B^\GG_{1/2}) \leq \liminf_{\varepsilon \rightarrow 0} |\nu_j * \varphi_\varepsilon - \nu|(B^\GG_{1/2}),
        \]
        thus for every $j \geq 1$ we can choose $\varepsilon_j \leq 1/j$ such that
        \begin{equation}
            \label{estimate on |nu_j - nu|}
            |\nu_j - \nu|(B^\GG_{1/2}) \leq |\nu_j * \varphi_{\varepsilon_j} - \nu|(B^\GG_{1/2}) + 1/j.
        \end{equation}
        Convolving (\ref{first calculation main theorem}) with $\varphi_{\varepsilon_j}$, we get
        \begin{equation}
            \label{second calculation main theorem}
            \mathscr{A}^k(\lambda u_j) = \mathscr{A}^k(V_j) - \sum_{h=1}^{k-1} \mathscr{A}^h \left( r_j^{k-h} c_j \left(T^{r_j}_\# \mu\right) * \varphi_{\varepsilon_j} \right),
        \end{equation}
        where $u_j := \nu_j * \varphi_{\varepsilon_j}$ and $V_j := (\lambda \nu_j -  c_j T^{r_j}_\# \mu) * \varphi_{\varepsilon_j}$ are smooth functions, $u_j \geq 0$ and\footnote{Viewing $u_j$ as a measure, i.e. identified with $u_j \Leb[d]$.}
        \begin{equation}
            \label{weak-* convergence of u_j}
            u_j \stackrel{*}{\rightharpoonup} \nu.
        \end{equation}
        
        By definition of $V_j$, $\lambda$, $\nu_j$ and $c_j$ and by classical convolution properties (e.g. see \cite[Theorem 2.2]{ambrosio-fusco-pallara} and \cite[Proposition 1.20]{Folland-Stein}), choosing $j$ large enough so that $\varepsilon_j \leq 1/4$ we notice that
        \begin{align*}
            \int_{B^\GG_{3/4}} |V_j| \ d\Leb[d](x) &\leq \frac{\left| \lambda T^{r_j}_\# |\mu|^s - T^{r_j}_\# \mu \right|(B^\GG_1)}{|\mu|^s(B^\GG_{r_j})} \\
            &= \frac{\left| \lambda |\mu|^s - (\mu^a + \mu^s) \right|(B^\GG_{r_j})}{|\mu|^s(B^\GG_{r_j})} \\
            &\leq \frac{\left| \lambda |\mu|^s - \mu^s \right|(B^\GG_{r_j})}{|\mu|^s(B^\GG_{r_j})} + \frac{|\mu|^a(B^\GG_{r_j})}{|\mu|^s(B^\GG_{r_j})} \\
            &= \fint_{B^\GG_{r_j}} \left| \frac{d\mu}{d|\mu|}(0) - \frac{d\mu}{d|\mu|}(x) \right| \ d|\mu|^s(x) + \frac{|\mu|^a(B^\GG_{r_j})}{|\mu|^s(B^\GG_{r_j})}.
        \end{align*}
        Therefore, $(i)$ yields
        \begin{equation}
            \label{limit L^1 norm of V_j}
            \lim_{j \rightarrow \infty} \int_{B^\GG_{3/4}} |V_j| \ d\Leb[d](x) = 0.
        \end{equation}
        
        Now localize via a positive cut-off function $\eta \in \mathscr{D}(B^\GG_{3/4})$ with $\chi_{B^\GG_{1/2}} \leq \eta \leq \chi_{B^\GG_{3/4}}$, writing \eqref{second calculation main theorem} as
        \begin{equation}
            \label{third calculation main thm}
            \mathscr{A}^k(\lambda \eta u_j) = \mathscr{A}^k(\eta V_j) + R_j,
        \end{equation}
        where the remainder term
        \[
        R_j(x) = \sum_{|\alpha| \leq k-1} b_\alpha(x) X^\alpha z^\alpha_j(x)
        \]
        is a finite sum of left-invariant smooth-coefficient (i.e. $b_\alpha \in \mathscr{D}(B^\GG_{3/4})$) differential operators of order at most $k-1$ applied to smooth and compactly supported functions $z^\alpha_j$ which satisfy $\sup_j \norm[z^\alpha_j][L^1(B^\GG_{3/4})] < \infty$, thanks to \eqref{weak-* convergence of u_j} and \eqref{limit L^1 norm of V_j}: more precisely\footnote{The first two terms come from Leibniz's rule, while the third one comes from the last term in \eqref{second calculation main theorem}.}
        \begin{align*}
            R_j = &\sum_{|\alpha| = k} \sum_{\underset{|\gamma| \geq 1}{\beta+\gamma = \alpha}} c_{\beta,\gamma} X^\gamma \eta X^\beta (A_\alpha \lambda \widetilde{\eta} u_j) \\
            + &\sum_{|\alpha| = k} \sum_{\underset{|\gamma| \geq 1}{\beta+\gamma = \alpha}} c_{\beta,\gamma} X^\gamma \eta X^\beta (A_\alpha \widetilde{\eta} V_j) \\
            + &\sum_{h=0}^{k-1} \sum_{|\alpha| = h} \eta X^\alpha (\widetilde{\eta} A_\alpha (r^{k-h} c_j T^{r_j}_\# \mu)*\varphi_{\varepsilon_j}),
        \end{align*}
        where $\widetilde{\eta}$ is another positive cut-off function such that $\chi_{\supp \eta} \leq \widetilde{\eta} \leq \chi_{B^\GG_1}$ to ensure that the $z^\alpha_j$'s are compactly supported. \\
        
        \emph{Step 2: Inverting through hypoellipticity}. To get meaningful information on $\nu_j$, we invert \eqref{third calculation main thm} exploiting $(iii)$. Applying $\mathcal{B}^*_\lambda$ on both sides we obtain
        \[
        \mathcal{B}^*_\lambda \mathcal{B}_\lambda(\eta u_j) = \mathcal{B}^*_\lambda \mathscr{A}^k(\eta V_j) + \mathcal{B}^*_\lambda R_j
        \]
        Thanks to Lemma \ref{lemma fundamental solution}, $\mathcal{B}^*_\lambda \mathcal{B}_\lambda$ admits a fundamental solution $K \in \mathscr{S}'(\GG)$, so that
        \begin{align*}
            \eta u_j = \left[ \mathcal{B}^*_\lambda \mathscr{A}^k(\eta V_j) + \mathcal{B}^*_\lambda R_j \right] * K &=: T_V(\eta V_j) + T_R(R_j) \\
            &=: f_j + g_j,
        \end{align*}
        and, since $\mathscr{A}^k$ and $\mathcal{B}_\lambda$ are homogeneous of the same order, Lemma \ref{lemma singular integrals} gives
        \[
        \norm[f_j][L^{1,\infty}] \lesssim \norm[V_j][L^1(B^\GG_{3/4})].
        \]
        In particular, \eqref{limit L^1 norm of V_j} yields
        \begin{gather}
            \label{f_j to zero in measure}
            f_j \rightarrow 0 \quad \text{in } \Leb[d]\text{-measure}; \\
            \label{f_j to zero in distribution}
            f_j \overset{*}\rightharpoonup 0 \quad \text{in } \mathscr{D}'(\GG),
        \end{gather}
        since $\langle f_j, \varphi \rangle = \langle \eta V_j, T_V^*(\varphi) \rangle \rightarrow 0$ for any $\varphi \in \mathscr{D}(\GG)$.
        
        To conclude the argument, we claim that
        \begin{equation}
            \label{precompactness of g_j}
            \{g_j\}_j \text{ is precompact in } L^1_\loc(\GG)
        \end{equation}
        and in particular it is equi-integrable. If so, $\eta u_j \geq 0$ yields
        \[
        f_j^- := \max\{0,-f_j\} \leq |g_j|
        \]
        and $\{f_j^-\}_j$ is equi-integrable too. Together with \eqref{f_j to zero in measure}, \eqref{f_j to zero in distribution} this implies that $f_j \rightarrow 0$ in $L^1_\loc(\GG)$ by Lemma \ref{corollary Vitali conv thm}. This implies the precompactness of $\{\eta u_j\}_j$ in $L^1_\loc(\GG)$, which means that it converges up to an (unlabeled) subsequence:  then, \eqref{weak-* convergence of u_j} implies that
        \[
        \eta \nu \in L^1(\GG) \quad \text{and} \quad \eta u_j \rightarrow \eta \nu \text{ in } L^1(\GG),
        \]
        which implies both claims \eqref{claim 1 main theorem}, \eqref{claim 2 main theorem} if one recalls \eqref{estimate on |nu_j - nu|}. This would conclude the proof provided that claim \eqref{precompactness of g_j} holds.
        
        \emph{Step 3: Precompactness}. Let us prove the claim: write $g_j$ as a finite sum of terms of the form
        \[
        g_j^\alpha = T_R \left( b_\alpha X^\alpha z^\alpha_j \right) = Q \circ (1 + \mathcal{R})^{-\frac{k}{\theta}} \circ P_\alpha \circ (1+\mathcal{R})^{\frac{|\alpha|-k}{\theta}} (z^\alpha_j)
        \]
        over $|\alpha| \leq k-1$. Note that $\mathcal{R}$ denotes the positive $\theta$-homogeneous Rockland operator of $\GG$, and  $(1 + \mathcal{R})^{-\frac{\beta}{\theta}}$ are Bessel potential operators, for $\beta > 0$ as in Definition \ref{definition bessel potentials}. Also, $Q$ and $P_\alpha$ are the pseudo-differential operators (see \cite[§5.1, §5.2]{Fischer-Ruzhansky}) of symbols, respectively,
        \begin{align*}
            \sigma_Q(\Pi) &= \Pi\left(\left[\mathcal{B}^*_\lambda \mathcal{B}_\lambda\right]^{-1}\right) \Pi(\mathcal{B}^*_\lambda) \Pi(1 + \mathcal{R})^{\frac{k}{\theta}} \quad \text{and} \\
            \sigma_{P_\alpha}(x,\Pi) &= b_\alpha(x) \Pi(X^\alpha) \Pi(1 + \mathcal{R})^{\frac{k-|\alpha|}{\theta}}
        \end{align*}
        for any $\Pi \in \widehat{\GG}$. Here $\left[\mathcal{B}^*_\lambda \mathcal{B}_\lambda\right]^{-1}$ denotes the right-convolution operator of kernel $K$, and
        \[
        \Pi(1 + \mathcal{R})^{\frac{\beta}{\theta}} := \Pi\left((1 + \mathcal{R})^{\frac{-\beta}{\theta}}\right)^{-1}
        \]
        are defined for $\beta > 0$ as in \cite[Lemma 5.1.2]{Fischer-Ruzhansky}.
        
        Now notice that $Q$ and $(1 + \mathcal{R})^{-\frac{k}{\theta}} \circ P_\alpha$ are pseudo-differential operators of order $0$, and therefore they extend continuously from $L^p(\GG)$ to itself for any $1 < p < \infty$: see \cite[Corollary 5.4.20]{Fischer-Ruzhansky}. On the other hand, Lemma \ref{lemma bessel potentials} ensures that $(1+\mathcal{R})^{\frac{|\alpha|-k}{\theta}}$ is a compact operator from $L^1_c(B^\GG_1)$ to $L^p(\GG)$ for any $1 \leq p < p(k-|\alpha|, Q)$. Thus, by uniform $L^1$-boundedness of $\{z^\alpha_j\}_j$ we conclude that the family $\{g_j^\alpha\}_j$ is precompact in $L^1_\loc(\GG)$.
    \end{proof}
    
\section{Proof of the Rank-one Theorem}\label{sec_rk1}
    We use the same notation introduced in Section~\ref{subsec_Carnot}; in particular, $\GG$ is a Carnot group of rank $\rho$ and topological dimension $d$. Given an open set $\Omega \subseteq \GG$ and a measurable function $u:\Omega\to\RR$, we denote by $D_H u=(X_1u,\dots,X_\rho u)$ its horizontal derivatives in the sense of distributions. We say that $u \in L^1_\loc(\Omega)$ is a function of \emph{locally bounded horizontal variation} in $\Omega$, or $u \in BV_{H,\loc}(\Omega)$, if i $D_H u$ is a vector-valued Radon measure. In this case
    \[
    \int_\Omega \varphi \ dD_Hu = - \int_\Omega u \diverg_H \varphi \ d\Leb[d] \qquad \text{for all } \varphi \in C^1_c(\Omega; \RR^\rho),
    \]
    where $\diverg_H\varphi := \sum_{i=1}^\rho X_i\varphi_i$ and we are implicitly using the fact that each $X_i$ is self-adjoint.
    If, in addition, $u \in L^1(\Omega)$ and $D_Hu$ has finite total variation on $\Omega$, we say that $u$ is of \emph{bounded horizontal variation} in $\Omega$, or $u \in BV_H(\Omega)$.

In view of Theorem~\ref{rank-one thm} we consider a vector-valued $BV_{H,loc}$ function, i.e., a map $u=(u_1,\dots,u_h):\Omega\to\RR^h$ such that $u_i\in BV_{H,\loc}(\Omega)$ for every $i=1,\dots,h$ or, equivalently, such that $D_H u$ is a $(h\times\rho)$-matrix-valued Radon measure. We will exploit Theorem \ref{main theorem} by finding appropriate necessary differential conditions for horizontal gradients.

    \subsection{Necessary constraints for horizontal gradients} \label{necessary constraints}
    Let $\GG$ be a Car\-not group of rank $\rho$ and step $s$. As usual, we fix an open subset $\Omega \subseteq \GG$. To understand which differential constraints are satisfied by horizontal gradients of $BV_H$ functions we exploit the algebraic structure of $\g$: let $X_1, \ldots, X_\rho$ be a basis for the first step of this algebra. For every multi-index $I=(i_1,\dots,i_z) \in \{1,\ldots,\rho\}^z$ of length $z$ we are going to define an operator $X_I$: this is done by induction on $z=|I|$. Clearly, if $z=1$ we set $X_I=X_{i_1}$, while for each $(z+1)$ multi-index $(i,I)=(i,i_1,\dots, i_z)$ we set
    \[
    X_{(i,I)}:=[X_{i},X_{I}]\,.
    \]
            By definition, $X_I = 0$ whenever $|I|=s+1$. This leads us to construct the family of differential operators\footnote{Throughout this chapter, the superscripts $z$ in $\mathscr{A}^z$ are purely notational, and they do not indicate the $z$-homogeneous part, as introduced in §\ref{sec_mainresult}, of some other operator $\mathscr{A}$. This nonetheless creates no contradictions, because the operators $\mathscr{A}^z$ we are constructing are themselves $z$-homogeneous.}
    \[
    \mathscr{A}^z : \mathscr{D}'(\Omega; \RR^{h \times \rho}) \rightarrow \mathscr{D}'(\Omega; \RR^{h \times \rho^{z+1}}),\qquad z\geq 1
    \]
    by setting iteratively, for every $w=(w^1, \ldots, w^\rho)\in \mathscr{D}'(\Omega; \RR^{h \times \rho})$, $j=1,\ldots,h$ (indicating the row), $i,k = 1, \ldots, \rho$ and $|I|=z+1$,
    \begin{align*}
    \mathscr{A}^1(w)^j_{(i,k)} &:= -X_k w^j_i + X_i w^j_k,\\
        \mathscr{A}^{z+1}(w)^j_{(i,I)} &:= X_i \mathscr{A}^z(w)^j_I - X_I w^j_i.
    \end{align*}
    Notice that, for any $\RR^h$-valued $BV_{H,\loc}$ function $u$ on $\Omega$, by definition we have $\mathscr{A}^1(D_H u)^j_{(i,k)} = X_{(i,k)} u$ and consequently 
    \[
    \mathscr{A}^{z+1}(D_H u)^j_{(i,I)} = X_iX_I u - X_IX_i u = X_{(i,I)} u
    \]
    for any $z \geq 1$. When $z=s$ this implies $\mathscr{A}^s(D_H u) = 0$,  which is the desired differential constraint:
    
    \begin{proposition}
        Let $(\mathscr{A}^z)_{z\geq 1}$ be the family of homogeneous left-invariant differential operators introduced above. Then for any $\RR^h$-valued $BV_{H,\loc}$ function $u$ on $\Omega$, $D_H u$ is an $\mathscr{A}^s$-free measure.
    \end{proposition}

    The proof of Theorem ~\ref{rank-one thm} relies on the hypoellipticity properties of $\mathscr{A}^s$ that will be studied in §\ref{subsec_hypoell}; before passing to that a few comments are in order.

    \begin{remark}
        When $\GG$ has step 1, i.e. $\rho=d$ and $\GG=\RR^d$, the necessary condition that $D_Hu=Du$ (i.e., the classical distributional derivatives of $u$) is an $\mathscr A^1$-free measure is the classical curl-free condition
        \[
            \mathscr{A}^1(Du)^j_{(i,k)} =  -X_kX_iu^j+X_iX_ku^j=0,\quad i,k=1\dots,d,j=1,\dots,h,
        \]
        since in this case one can choose the standard basis $X_i=\partial_{x_i}$ for every $i=1,\dots,d$.
    \end{remark}
        
    \begin{remark}\label{rem_rumin}
        Let us consider the case $\GG=\HH^1$: here, derivatives of $BV_H$ are $\mathscr A^2$-free and, using the more standard notation $X,Y$ in place of (respectively) $X_1,X_2$, it can be checked that
        \begin{align*}
        &\mathscr{A}^2(w)^j_{(1,1,1)}=\mathscr{A}^2(w)^j_{(1,2,2)}=\mathscr{A}^2(w)^j_{(2,1,1)}=\mathscr{A}^2(w)^j_{(2,2,2)}=0,
            \\
            &\mathscr{A}^2(w)^j_{(1,1,2)}=-\mathscr{A}^2(w)^j_{(1,2,1)}=( - 2XY+YX)w_1^j+XXw_2^j,\\
            &\mathscr{A}^2(w)^j_{(2,1,2)}=-\mathscr{A}^2(w)^j_{(2,2,1)}=-YYw_1^j+(2YX-XY)w^j_2.
        \end{align*}
        It follows that, apart from several redundancies, the operator $\mathscr A^2$ is the Heisenberg curl operator that can be introduced in terms of {\em Rumin's differential complex}, see e.g. \cite[Example 3.11]{baldi_franchi}.
        
        On the contrary, when $n\geq 2$ the operator $\mathscr A^2$ in $\HH^n$ is different from Rumin's Heisenberg curl: in fact, the latter is here a first-order operator, while $\mathscr A^2$ is a second-order one. We will introduce in Remark~\ref{rem_altrecondnecinH^2} alternative necessary conditions for derivatives of $BV_H$ functions that can be used to prove Theorem~\ref{rank-one thm} in $\HH^n$, $n\geq 2$, via Theorem~\ref{main theorem}; these necessary conditions will be provided by an operator $\mathscr R$ that, again, is in correspondence with Rumin's differential operator in $\HH^n$.
    \end{remark}

    \begin{remark}\label{rem_dueconiHeisenberg}
        In order to prove Theorem~\ref{rank-one thm} in the first Heisenberg group $\HH^1$ one cannot make use of the  (Euclidean) structure result for $\mathscr{A}^2$-free measures in \cite{DPR}: it was in fact shown in \cite[Remark~5.6]{don-massaccesi-vittone} that the wave cone of $\mathscr{A}^2$  in $\HH^1$ is the full space $\RR^{h \times 2}$.

        As already mentioned, we will introduce later (see Remark \ref{rem_altrecondnecinH^2}) an operator $\mathscr R$ providing other necessary conditions for horizontal gradients in $\HH^n$, $n\geq 2$. Again, Theorem~\ref{rank-one thm} in  $\HH^n$ does not follow from \cite{DPR} (applied to $\mathscr R$) because  the wave cone of $\mathscr R$  is  the full space $\RR^{h \times 2n}$. See \cite[Remark~5.5]{don-massaccesi-vittone}.
    \end{remark}

    \subsection{Hypoellipticity of the necessary conditions}\label{subsec_hypoell}
    
    We can now use the results and notation of Section~\ref{sec_mainresult}, with $k=s$, $m=h\rho$ (identifying elements of $\RR^{h\rho}$ with matrices in $\RR^{h\times\rho}$) and $n=h\rho^{s+1}$. Following Definition \ref{definition hypoelliptic wavecone}, for every $\lambda \in \RR^{h \times \rho}$ we consider the operator $\mathcal{B}^z_\lambda : \mathscr{D}'(\Omega) \rightarrow \mathscr{D}'(\Omega; \RR^{h\times \rho^{z+1}})$ defined by
    \[
    \mathcal{B}^z_\lambda(\varphi) := \mathscr{A}^z(\lambda \varphi).
    \]
    and its adjoint  $(\mathcal{B}^z_\lambda)^*$ defined by duality as
    \[
    \langle (\mathcal{B}^z_\lambda)^*(\xi), \varphi \rangle = \langle \xi, \mathcal{B}^z_\lambda(\varphi) \rangle,\qquad \xi \in \mathscr{D}(\Omega; \RR^{h\times \rho^{z+1}}),
    \]
    where here $\langle\cdot,\cdot\rangle$ is the standard scalar product on $L^2(\Omega)$ (on the left-hand side) and on $L^2(\Omega;\RR^{h\times \rho^{z+1}})$ (on the right-hand side).

    \begin{proposition}
        \label{hypoellipticity}
        For any $z \geq 1$, $(\mathcal{B}^z_\lambda)^* \mathcal{B}^z_\lambda$ is hypoelliptic if and only if $\lambda$ is of rank at least $2$. In other words,
        \[
            \Lambda^H_{\mathscr{A}^z} = \{\lambda \in \RR^{h \times \rho} : \operatorname{rank}(\lambda) \leq 1\}.
        \]
    \end{proposition}
    \begin{proof}
        Write $\lambda = (\lambda^1, \ldots, \lambda^h) \in \RR^{h \times \rho}$ and, for any $z \geq 1$, denote for the sake of convenience $\mathcal{B}^z_{\lambda^j,I} := (\mathcal{B}^z_\lambda)^j_{I}$ for any $j = 1, \ldots, h$ and $|I|=z+1$. Then we may write
        \begin{equation}
            \label{B^z sum decomposition}
            (\mathcal{B}^z_\lambda)^*\mathcal{B}^z_\lambda = \sum_{j=1}^h \sum_{|I|=z+1} (\mathcal{B}^z_{\lambda^j,I})^*\mathcal{B}^z_{\lambda^j,I}.
        \end{equation}
        By Theorem \ref{thm_rockland}, the hypoellipticity of $(\mathcal{B}^z_\lambda)^* \mathcal{B}^z_\lambda$ is equivalent to the fact that for any representation $\Pi \in \widehat{\GG}$ the linear map $\Pi[(\mathcal{B}^z_\lambda)^*\mathcal{B}^z_\lambda]$ is injective, i.e.,
        \begin{equation}
            \label{injec for z}
            \langle \Pi[(\mathcal{B}^z_\lambda)^*\mathcal{B}^z_\lambda] \psi, \psi \rangle \not= 0 \qquad \text{for any } C^\infty \text{-vector } \psi \not= 0.
        \end{equation}
        Notice that by \eqref{B^z sum decomposition}
        \begin{align}
            \label{B^z represented sum decomposition}
            \notag \langle \Pi[(\mathcal{B}^z_\lambda)^*\mathcal{B}^z_\lambda] \psi, \psi \rangle &= \sum_{j=1}^h \sum_{|I|=z+1} \langle \Pi[(\mathcal{B}^z_{\lambda^j,I})^*\mathcal{B}^z_{\lambda^j,I}] \psi, \psi \rangle \\
            &= \sum_{j=1}^h \sum_{|I|=z+1} \|\Pi[\mathcal{B}^z_{\lambda^j,I}] \psi\|^2,
        \end{align}
        which combined with \eqref{injec for z} implies that the hypoellipticity of $(\mathcal{B}^z_\lambda)^* \mathcal{B}^z_\lambda$ is equivalent to the following:
        \begin{equation}
            \label{eq_equivalent_hypoell_indices}
             \forall \Pi \in \widehat{\GG}, \psi \ C^\infty \text{-vector} \quad \exists j,I: \ ~ \Pi[\mathcal{B}^z_{\lambda^j,I}] \psi \not= 0.
        \end{equation}\medskip       
        
        Suppose first that $\lambda$ is of rank at least 2: we rely on ~\eqref{eq_equivalent_hypoell_indices} to show the hypoellipticity of $(\mathcal{B}^z_\lambda)^*\mathcal{B}^z_\lambda$. We first present the case in which $z=1$, and then argue by reduction to the base case.
        
        \textit{Base case}. If $z=1$, recall that for any $i,k = 1, \ldots, \rho$
        \[
        \mathcal{B}^1_{\lambda^j,(i,k)} = \lambda^j_k X_i - \lambda^j_i X_k.
        \]
        If, by contradiction, \eqref{eq_equivalent_hypoell_indices} does not hold for $z=1$, then there  exist a  representation $\Pi \in \widehat{\GG}$ and a $C^\infty$-vector $\psi \not= 0$ such that
        \begin{equation}\label{eq_treno}
        \lambda^j_k\Pi[ X_i ]\psi- \lambda^j_i\Pi[ X_k]\psi =
        \Pi[\lambda^j_k X_i - \lambda^j_i X_k]\psi = 0
        \end{equation}
        for any row $j$ and any $i,k$. Condition~\eqref{eq_treno} implies that, for every $j$, the row $\lambda^j$ and  $v:=(\Pi[X_1]\psi,\dots,\Pi[X_\rho]\psi)$ are linearly dependent; this is possible only if $v=0$, because $\lambda$ has rank at least 2 and thus one can find linearly independent rows $\lambda^{j_1},\lambda^{j_2}$. However, the fact that $\Pi[X_i]\psi=0$ for every $i$ contradicts the hypoellipticity of the sub-Laplacian $X_1^2+\dots+X_\rho^2$.
        
        \textit{General case}. Now suppose that \eqref{eq_equivalent_hypoell_indices} fails at step $z+1$ (with $z \geq 1$). Then there exist some representation $\Pi \in \widehat{\GG}$ and some $C^\infty$-vector $\psi \not= 0$ such that
        \begin{equation}
            \label{eq_contrad_hyp}
            \Pi[\mathcal{B}^{z+1}_{\lambda^j, (i_0,I)}] \psi = 0
        \end{equation}
        for any $j,I=(i_1, \ldots, i_{z+1})$ as above and $i_0 = 1, \ldots, \rho$. Now, to reduce to the base case, recall that by definition
        \begin{equation}
            \label{B_lambda recursive def}
            \mathcal{B}^{z+1}_{\lambda^j,(i_0,I)} = X_{i_0} \mathcal{B}^z_{\lambda^j,I} - \lambda^j_{i_0} X_I,
        \end{equation}
        therefore \eqref{eq_contrad_hyp} yields
        \begin{equation}
            \label{inductive identity}
            \Pi[X_{i_0} \mathcal{B}^z_{\lambda^j,I}] \psi = \lambda^j_{i_0} \Pi[X_I] \psi.
        \end{equation}
        Now for any $1 \leq k_0 \leq \rho$, multiply \eqref{inductive identity} by $\lambda^j_{k_0}$ and then subtract the same expression with $i_0$ and $k_0$ swapped, finding
        \[
            \Pi[\lambda^j_{k_0} X_{i_0} - \lambda^j_{i_0} X_{k_0}] \Pi[\mathcal{B}^z_{\lambda^j,I}] \psi = \Pi[\mathcal{B}^1_{\lambda^j, (i_0,k_0)}] \Pi[\mathcal{B}^z_{\lambda^j,I}] \psi = 0.
        \]
        Now we repeat the same computation: we write $I=(i_1,I')$ with $I' = (i_2, \ldots, i_{z+1})$ and expand $\mathcal{B}^z_{\lambda^j,I}$ to find
        \[
\Pi[\mathcal{B}^z_{\lambda^j,I}] \psi=\Pi[X_{i_1} \mathcal{B}^{z-1}_{\lambda^j,I'}] \psi - \lambda^j_{i_1} \Pi[X_{I'}] \psi   .
        \]
Left-composing the previous equality  with $\lambda^j_{k_1} \Pi[\mathcal{B}^1_{\lambda^j, (i_0,k_0)}]$, where $1 \leq k_1 \leq \rho$, and subtracting the same expression with $i_1$ and $k_1$ swapped, one achieves
        \[
            \Pi[\mathcal{B}^1_{\lambda^j, (i_0,k_0)}] \Pi[\mathcal{B}^1_{\lambda^j, (i_1,k_1)}] \Pi[\mathcal{B}^{z-1}_{\lambda^j,I'}] \psi = 0.
        \]
        Iterating this procedure one gets
        \begin{equation}
            \label{eq_contrad}
             \Pi[\mathcal{B}^1_{\lambda^j, (i_0,k_0)}]\ \Pi[\mathcal{B}^1_{\lambda^j, (i_1,k_1)}]\ \dots\ \Pi[\mathcal{B}^1_{\lambda^j, (i_{z+1},k_{z+1})}] \psi = 0,
        \end{equation}
        for any choice of $1 \leq i_a, k_a \leq \rho$. Now take $i_0 = i_1 = \cdots = i_{z+1} =: i$ and $k_0 = k_1 = \cdots = k_{z+1} =: k$. Then, \eqref{eq_contrad} and the skew-adjointness of $\mathcal{B}^1_{\lambda^j,(i,k)}$ negate \eqref{eq_equivalent_hypoell_indices} for $z=1$, consequently negating the hypoellipticity of $(\mathcal{B}^1_\lambda)^*\mathcal{B}^1_\lambda$, which is impossible. The proof of the first implication is accomplished.\medskip
        
        Conversely, suppose that $\lambda$ is of rank at most $1$: since all rows $\lambda^j$ are collinear, without loss of information we drop the superscript $j$ for the sake of convenience. We proceed by induction as above, this time to prove non-hypoellipticity.
        
        \textit{Base case}. Let $\Pi$ be the following irreducible representation, which is degenerate and 1-dimensional in the sense that it collapses all strata of the algebra except for the first horizontal one\footnote{These degenerate representations correspond to unitary irreducible representations of $\RR^\rho$, i.e. the commutative framework of classical Fourier analysis.}, acting by multiplication as
        \[
            \Pi(x) = e^{i(\lambda_1x_1+\cdots+\lambda_\rho x_\rho)}
        \]
        for any $x \in \GG \equiv \RR^d$. Indeed it is easily checked that at the Lie algebra level $\Pi$ acts as
        \begin{align}
            \notag \Pi[X_j]\psi &= i\lambda_j\cdot \psi &&\text{for all } j=1, \ldots, \rho; \\
            \label{degenerate representation} \Pi[X_I]\psi &= 0 &&\text{whenever } |I|\geq 2.
        \end{align}
        Note that here $i=\sqrt{-1}$. For this particular representation, \eqref{eq_treno} is trivially satisfied for any non-zero $C^\infty$-vector $\psi$, meaning that $\Pi[\mathcal{B}^1_{\lambda,(j,k)}]=0$ for any choice of $j,k$. Therefore $(\mathcal{B}^1_\lambda)^*\mathcal{B}^1_\lambda$ is not hypoelliptic.

        \textit{Inductive case}. Suppose $\Pi[\mathcal{B}^z_{\lambda,I}]=0$ for any choice of $|I|=z+1$. Then with \eqref{B_lambda recursive def} and \eqref{degenerate representation}, we immediately conclude that
        \[
            \Pi[\mathcal{B}^{z+1}_{\lambda,(j,I)}]\psi = \Pi[X_j \mathcal{B}^z_{\lambda,I}]\psi = 0
        \]
        for any choice of $j$. Therefore $(\mathcal{B}^{z+1}_\lambda)^*\mathcal{B}^{z+1}_\lambda$ is not hypoelliptic. This concludes the proof of the second implication.
    \end{proof}

    \begin{remark}\label{rem_altrecondnecinH^2}
        In order to prove Theorem~\ref{rank-one thm} we used the necessary conditions provided by $\mathscr A^s(D_Hu)=0$; depending on the group $\GG$, however, one might consider also different operators. For instance, in the second Heisenberg group\footnote{Using the notation introduced in §\ref{subsec_heisenberg}.} $\HH^2$ (but the argument can be easily adapted to any $\HH^n$, $n\geq 2$) one can observe that the equalities
        \[
        [X_1,X_2]u=[X_1,Y_2]u=[X_2,Y_1]u=[Y_1,Y_2]u=([X_1,Y_1]-[X_2,Y_2])u=0
        \]
        (to be read as $X_1X_2u-X_2X_1u=0$, etc.) imply that every distributional derivative $D_Hu$ is  $\mathscr R$-free, where for a $h\times 4$-valued $w$ we define
        \begin{align*}
            &\mathscr R(w)^j_1:=X_1w_2^j-X_2w_1^j\\
            &\mathscr R(w)^j_2:=X_1w_4^j-Y_2w_1^j\\
            &\mathscr R(w)^j_3:=X_2w_3^j-Y_1w_2^j\\
            &\mathscr R(w)^j_4:=Y_1w_4^j-Y_2w_3^j\\
            &\mathscr R(w)^j_5:=X_1w_3^j-Y_1w_1^j-X_2w_4^j+Y_2w_2^j.
        \end{align*}
        It is worth noticing that $\mathscr R$ provides the Heisenberg curl operator in $\HH^n,\ n\geq 2$, as per the Rumin's complex; see e.g. ~\cite[Example 3.12]{baldi_franchi}.
        
        Writing $\mathcal B_\lambda(\varphi):=\mathscr R(\lambda\varphi)$ one can check that
        \begin{align*}
            (\mathcal B_\lambda)^*\mathcal B_\lambda&=-\sum_{j=1}^h  \Big[(\lambda^j_2X_1-\lambda^j_1X_2)^2 + (\lambda^j_4X_1-\lambda^j_1Y_2)^2 
        +(\lambda^j_3X_2-\lambda^j_2Y_1)^2\\
        &\hphantom{=-\sum_{j=1}^h  \Big[}+(\lambda^j_4Y_1-\lambda^j_3Y_2)^2
        + (\lambda_3^jX_1-\lambda_1^jY_1-\lambda_4^jX_2+\lambda_2^jY_2)^2\Big]\\
        &=:-\sum_{j=1}^h \Big[(Z_1^j)^2 + (Z_2^j)^2 + (Z_3^j)^2 + (Z_4^j)^2 + (Z_5^j)^2\Big]\,. 
        \end{align*}
        Consider the standard scalar product on $\mathfrak h_1$ making $X_1,X_2,Y_1,Y_2$ orthonormal. 
        It is straightforward to check that, if $\lambda^j\neq0$, then  $Z_1^j,\dots,Z_5^j$ linearly generate the 3-dimensional subspace orthogonal to $W^j:=\lambda^j_1 X_1+\lambda^j_2X_2+\lambda^j_3 Y_1+\lambda^j_4 Y_2$. 
        If now  $\lambda^{j_1},\lambda^{j_2}$ are linearly independent, so are $W^{j_1}$ and $W^{j_2}$, hence
        \begin{itemize}
            \item[(a)] the vectors $(Z_i^{j_1})_{i=1,\dots,5}$ and $(Z_i^{j_2})_{i=1,\dots,5}$ linearly generate the space $(W^{j_1})^\perp +(W^{j_2})^\perp =\mathfrak h_1 $;
            \item[(b)] at least two vectors among  $Z^{j_1}_1,\dots,Z^{j_1}_5$ do not (Lie bracket) commute, because a 3-dimensional subspace of $\mathfrak h_1$ is never commutative.
        \end{itemize}
        It follows that $(Z_i^{j_1})_{i=1,\dots,5}$ and $(Z_i^{j_2})_{i=1,\dots,5}$ bracket-generate the whole algebra; by Theorem~\ref{thm_Hormanderdacitare}, the sum-of-squares operator $(\mathcal B_\lambda)^*\mathcal B_\lambda$ is hypoelliptic.
        
        Summarizing: if $\lambda$ has rank at least 2, then $\lambda\not\in\Lambda_{\mathscr R}^H$\footnote{Actually, it can be easily checked that $\Lambda_{\mathscr R}^H$ consists precisely of the matrices of rank at most 1.}, and Theorem~\ref{rank-one thm} in $\HH^n,\ n\geq 2$, follows again from Theorem~\ref{main theorem}.
    \end{remark}

    \appendix
    \section{A weak differentiation result in Carnot Groups}\label{app_Besicovitch}
    
    In this section we prove Theorem~\ref{weak besicovitch thm Carnot}; before this we need a couple of lemmas.
    
    \begin{lemma}[Weak doubling property]
        \label{weak doubling lemma}
        Let $\nu$ be a positive Radon measure on a Carnot group $\GG \equiv \RR^d$ of step $s$, and pick $t \in (0,1)$. Then there exists a measurable set $E^t_\nu$ such that $\nu(E^t_\nu)=0$ and
        \begin{equation}
            \label{weak doubling}
            \limsup_{r \downarrow 0} \frac{\nu(B^\GG(x,tr))}{\nu(B^\GG(x,r))} > 0 \quad \text{for all } x \in \supp(\nu) \backslash E^t_\nu,
        \end{equation}
        where $B^\GG(x,r)$ denotes the ball centered at $x \in \GG$ of radius $r > 0$ with respect to the Carnot-Carathéodory distance $d_c$.
    \end{lemma}
    \begin{proof}
        Let $E^t_\nu$ be the $\nu$-measurable set where
        \[
        \lim_{k \rightarrow \infty} \frac{\nu(B^\GG(x,t^{k+1}))}{\nu(B^\GG(x,t^k))} = 0
        \]
        and notice that \eqref{weak doubling} is satisfied outside of $E^t_\nu$. Now assume by contradiction that $\nu(E^t_\nu) > 0$: then, since the functions
        \[
        \frac{\nu(B^\GG(x,t^{k+1}))}{\nu(B^\GG(x,t^k))}
        \]
        are measurable, by Egorov's Theorem there exists a compact $K \subseteq E^t_\nu$ such that $\nu(K) > 0$ on which the limit is uniform. In other words, for any $\varepsilon > 0$ there exists some index $k_0$ such that
        \[
        \nu(B^\GG(x,t^{k+1})) \leq \varepsilon \nu(B^\GG(x,t^k))
        \]
        for all $x \in K$ and $k \geq k_0$, therefore
        \[
        \nu(B^\GG(x,t^k)) \leq \frac{\nu(B^\GG(x,t^{k_0}))}{\varepsilon^{k_0}} \varepsilon^k.
        \]
        By the compactness of $K$, the fact that $t < 1$ and that $\nu$ is Radon, the fraction on the right hand side is uniformly bounded in $x \in K$. Hence by Corollary \ref{estimate d_c} we gather that
        \[
        \nu(B^{\RR^d}(x,t^{sk})) \leq C \varepsilon^{k-k_0}
        \]
        for some constant $C > 0$ ($B^{\RR^d}$ indicates a ball w.r.t. the euclidean norm).
        
        Now choose  $\varepsilon =  t^{sd}/2$: we obtain the sequence $r_k := t^{sk}$ which satisfies
        \[
        \nu(B^{\RR^d}(x,r_k)) \leq C 2^{-(k-k_0)} \Bigl(\frac{r_k}{r_{k_0}}\Bigr)^d.
        \]
        For every $k_1>k_0$ we apply the Vitali-Besicovitch Covering Theorem to the family of Euclidean balls
        \[
        \mathcal{F}_{k_1} := \left\{\overline{B^{\RR^d}(x,r_k)}\right\}_{x \in K, k \geq k_1}
        \]
        to obtain a disjoint subcover $\mathcal{G}_{k_1} \subseteq \mathcal{F}_{k_1}$ such that
        \begin{align*}
        \nu(K) & = \sum_{B^{\RR^d}(x,r_k) \in \mathcal{G}_{k_1}} \nu(B^{\RR^d}(x,r_k))\\
        & \leq C\:\frac{2^{-(k_1-k_0)}}{r_{k_0}^d}  \sum_{B^{\RR^d}(x,r_k) \in \mathcal{G}_{k_1}} r_k^d\\
        &\leq C \:\frac{2^{-(k_1-k_0)}}{r_{k_0}^d} \Leb[d](K + B(0,1)) \longrightarrow 0 \quad \text{as } k_1 \to+\infty.
         \end{align*}
        This contradicts $\nu(K) > 0$, therefore $\nu(E^t_v) = 0$.
    \end{proof}
    
    Now let us define the sets
    \[
    E_j := \left\{ x \in \supp(\nu) \backslash E^t_\nu : \ \limsup_{r \downarrow 0} \frac{\nu(B^\GG(x,tr))}{\nu(B^\GG(x,r))} > \frac{1}{j} \right\}
    \]
    which cover $\supp(\nu) \backslash E^t_\nu$ completely and such that $\bigcup_j E_j$ has full $\nu$-measure by the lemma we have just shown. For any $x \in E_j$ we may pick a sequence $r_k^x$ of radii such that $\nu(B^\GG(x,tr_k^x)) \geq \frac{1}{j} \nu(B^\GG(x,r_k^x))$. Then for any $f \in L^1(\GG, \nu)$ we define the maximal function
    \[
    Mf: \bigcup_j E_j \longrightarrow \RR \qquad Mf(x) := \sup_{k \in \NN} \fint_{B^\GG(x,r_k^x/5)} |f| \ d\nu.
    \]
    Note that $Mf$ might not be $\nu$-measurable, because the dependence of $r_k^x$ on $x$ might be arbitrarily irregular. Nevertheless we can reach this next weak $(1,1)$ estimate:
    \begin{lemma}
        \label{weak estimate differentiation Carnot}
        Let $E_j$ and $r_k^x$ be defined as above. Then for any $\lambda > 0$ we have
        \[
        \nu^*(\{ Mf > \lambda \} \cap E_j) \leq \frac{j}{\lambda} \int_{\GG} |f| \ d\nu,
        \]
        where $\nu^*$ denotes the outer measure associated with $\nu$.
    \end{lemma}
    \begin{proof}
        Choose $t = 1/5$ and, for any $x \in \{ Mf > \lambda \} \cap E_j$, choose $k = k(x,\lambda)$ such that
        \[
        r_k^x \leq 1, \qquad \fint_{B^\GG(x,r_k^x/5)} |f| \ d\nu > \lambda, \qquad \frac{\nu(B^\GG(x,r_k^x/5))}{\nu(B^\GG(x,r_k^x))} > \frac{1}{j}
        \]
        by construction. Clearly the family $\mathcal{F} = \{B^\GG(x,r_{k(x,\lambda)}^x/5)\}_x$ covers $\{ Mf > \lambda \} \cap E_j$, hence the classical 5-covering theorem (see \cite[Theorem 1.2]{heinonen}) provides a disjoint subfamily $\mathcal{F}' \subseteq \mathcal{F}$ such that $\{5B\}_{B \in \mathcal{F}'}$ also covers $\{ Mf > \lambda \} \cap E_j$. Therefore
        \[
        \nu^*(\{ Mf > \lambda \} \cap E_j) \leq \sum_{B \in \mathcal{F}'} \nu(5B) \leq j \sum_{B \in \mathcal{F}'} \nu(B) \leq \frac{j}{\lambda} \int_{\GG} |f| \ d\nu,
        \]
        which is the conclusion.
    \end{proof}
    
    \begin{proof}[Proof of Theorem \ref{weak besicovitch thm Carnot}]
        Consider the following sets
        \[
        A_{n,j} := \left\{ x \in E_j : \ \limsup_{k \rightarrow \infty} \fint_{B^\GG(x,r_k^x)} |f-f(x)| \ d\nu > \frac{1}{n} \right\}
        \]
        and notice that for any $g \in C(\GG)$ we have
        \begin{align*}
            A_{n,j} &\subseteq \left\{ x \in E_j : \ \limsup_{k \rightarrow \infty} \fint_{B^\GG(x,r_k^x)} |f-g| \ d\nu > \frac{1}{2n} \right\} \cup\\
            &\qquad\cup\left\{ x \in E_j : \ |f(x)-g(x)| > \frac{1}{2n} \right\} \\
            &\subseteq \left\{ x \in E_j : \ M|f-g|(x) > \frac{1}{2n} \right\} \cup \left\{ x \in E_j : \ |f(x)-g(x)| > \frac{1}{2n} \right\}.
        \end{align*}
        By subadditivity of $\nu^*$, Lemma \ref{weak estimate differentiation Carnot} ensures that
        \[
        \nu^*(A_{n,j}) \leq 2nj \int_{\GG} |f-g| \ d\nu + 2n \int_{E_j} |f-g| \ d\nu,
        \]
        therefore it is sufficient to let $g \rightarrow f$ in $L^1(\GG)$ to obtain $\nu^*(A_{n,j}) = 0$. Moreover the statement \eqref{measure differentiation Carnot} fails on the set
        \[
        F := \bigcup_{n,j} A_{n,j} \cup \left( \supp(\nu) \backslash \bigcup_j E_j \right),
        \]
        which we have just shown to be of null $\nu^*$-measure: by the completeness of the $\sigma$-algebra of $\nu$-measurable sets, $F$ is a $\nu$-measurable null set, which is exactly the statement.
    \end{proof}
    
    \section{Existence of tangent measures}\label{app_misuretangenti}
    Here we prove Proposition~\ref{prop existence of tangent measures}. As in Section~\ref{subsec_tools}, we fix an open set $\Omega$ and a Radon measure $\mu \in \mathcal{M}_\loc(\Omega;V)$  taking values in a real finite dimensional vector space $V$. Given $x_0 \in \supp \mu$ and a sequence  $r_j \downarrow 0$ we write
    \[
    \nu_j := \frac{T^{x_0,r_j}_\#\mu}{|\mu|(B^\GG(x_0,r_j))}
    \]
    where $	T^{x_0,r_j}(x) := \delta_{1/r_j} (x_0^{-1} \cdot x)$.
    
    The following lemma is an adaptation of \cite[Theorem 2.44]{ambrosio-fusco-pallara}.
    
    \begin{lemma}
        \label{lemma tangent of the total variation}
        At $|\mu|$-almost every $x_0 \in \Omega$ there exists a sequence of radii $r_j^{x_0} = r_j \downarrow 0$ such that
        \[
        \nu_j \overset{*}{\rightharpoonup} \nu \quad \Longleftrightarrow \quad |\nu_j| \overset{*}{\rightharpoonup} |\nu|.
        \]
        If this is the case, $\nu = \frac{d\mu}{d|\mu|}(x_0) |\nu|$, and in particular
        \[
        \operatorname{Tan}(\mu,x_0) = \frac{d\mu}{d|\mu|}(x_0) \operatorname{Tan}(|\mu|,x_0) \qquad |\mu| \text{-a.e. in } \Omega.
        \]
    \end{lemma}
    \begin{proof}
        We apply Theorem~\ref{weak besicovitch thm Carnot} to $f := \frac{d\mu}{d|\mu|}$ and the Radon measure $|\mu|$, obtaining at $|\mu|$-almost every point $x_0 \in \Omega$ an infinitesimal sequence $r_j := r_j^{x_0}/R$ of radii (where $R > 0$ is to be chosen later) such that
        \begin{equation}
            \label{lebesgue point for polar}
            \lim_{j \rightarrow \infty} \fint_{B^{\GG}(x_0,r_jR)} |f-f(x_0)| d|\mu| = 0.
        \end{equation}
        Now pick any $\varphi \in C_c(\GG)$. Calling $c_j = |\mu|(B^\GG(x_0, r_jR))^{-1}$, and for $R > 0$ large enough so that $B^\GG(0,R) \Supset \supp \varphi$ we get
        \begin{align*}
            \hspace{40pt}&\hspace{-40pt} \int \varphi \ dT^{x_0,r_j}_\# |\mu| - \int \varphi \ f(x_0) \cdot dT^{x_0,r_j}_\# \mu = \\
            &= \int \varphi(y) \left[1 - \langle f(x_0) , f(x_0 \cdot \delta_{r_j}(y)) \rangle\right] \ dT^{x_0,r_j}_\# |\mu|(y) \\
            &= \int \varphi(\delta_{1/r_j}(x_0^{-1}x)) \left[1 - \langle f(x_0) , f(x) \rangle\right] \ d|\mu|(x),
        \end{align*}
        and since $|1 - \langle f(x_0), f(x) \rangle| \leq |f(x_0) - f(x)|$,
        \begin{align*}
            \hspace{40pt}&\hspace{-40pt} c_j \left| \int \varphi \ dT^{x_0,r_j}_\# |\mu| - \int \varphi \ f(x_0) \cdot dT^{x_0,r_j}_\# \mu \right| \\
            &\leq \norm[\varphi][\infty] \fint_{B^{\GG}(x_0,r_j R)} |f(x_0)-f(x)| \ d|\mu|(x) \longrightarrow 0
        \end{align*}
        as $j \rightarrow \infty$ by \eqref{lebesgue point for polar}.
        
        This shows that if $\nu_j \overset{*}{\rightharpoonup} \nu$, then $|\nu_j| \overset{*}{\rightharpoonup} f(x_0) \cdot \nu =: \sigma$. If we write $\nu = g |\nu|$ for some $g \in L^1_\loc(\GG, |\nu|; \mathbb{S}^{d-1})$, notice that
        \[
        |\nu| \leq \sigma = f(x_0) \cdot \nu = \langle f(x_0) , g \rangle |\nu|  \leq |\nu|
        \]
        where the first inequality follows from the weak-* lower semicontinuity of the total variation. This implies that $g \equiv f(x_0)$ for $|\nu|$-almost every point, and thus $\sigma = |\nu|$.
        
        Conversely, if $|\nu_j|$ converges weakly-* to some (positive) measure $\sigma \in \mathcal{M}_\loc(\GG)$ then (up to an unlabeled subsequence)
        \[
        \nu_j \overset{*}{\rightharpoonup} \nu
        \]
        for some $\nu \in \mathcal{M}_\loc(\GG;V)$. The same argument as above applies, and again we discover that $\sigma = |\nu|$.
    \end{proof}
    
    The proof of Proposition~\ref{prop existence of tangent measures} is now an almost verbatim adaptation of \cite[Lemma 10, Appendix A]{Rindler-lowersemicontinuity_for_integral_functionals}.

    \begin{proof}[Proof of Proposition~\ref{prop existence of tangent measures}]
        By the locality of the definition of tangent measure, taking a large enough ball around $x_0$ we may assume $\mu$ to be supported in some compact set $K \subseteq \Omega$ containing $x_0$. Moreover, by Lemma \ref{lemma tangent of the total variation} we restrict to the case in which $\mu$ is a positive real-valued measure.
        
        \noindent\emph{Step 1}. For all bounded Borel sets $A \subseteq \GG$ by Fubini
        \[
        \mu(A) = \frac{1}{\omega_\GG r^Q} \int_{\GG} \mu(A \cap B^\GG(x,r)) \ dx,
        \]
        where $\omega_\GG$ is the volume of $B^{\GG}(0,1)$ with respect to the Haar measure $\Haus[Q]$, $Q$ being the homogeneous dimension of $\GG$ as a Carnot group. Indeed
        \begin{align*}
            \int_{\GG} \mu(A \cap B^\GG(x,r)) \ dx &= \int_{\GG} \int_{\GG} \mathds{1}_{A}(y) \mathds{1}_{B^\GG(x,r)}(y) \ d\mu(y) dx \\
            &= \int_{\GG} \mathds{1}_{A}(y) \int_{\GG} \mathds{1}_{B^\GG(y,r)}(x) \ dx d\mu(y) = \omega_\GG r^Q \mu(A).
        \end{align*}
        
        \noindent\emph{Step 2}. For all $t > 1$ we show that
        \[
        \lim_{\beta \rightarrow \infty} \limsup_{r \downarrow 0} \mu(\{ x \in K : \mu(B^\GG(x,tr)) \geq \beta \mu(B^\GG(x,r)) \}) = 0.
        \]
        Indeed, let $\varepsilon > 0$ and fix $r > 0$. Defining, for a fixed $\beta > 0$,
        \[
        E = \left\{ x \in K : \mu(B^\GG(x,tr)) \geq \beta \mu(B^\GG(x,r)) \right\},
        \]
        notice that whenever $B^\GG(x,r/2) \cap E \not= \emptyset$, taking $z \in B^\GG(x,r/2) \cap E$ we can estimate
        \[
        \beta \mu(B^\GG(x,r/2)) \leq \beta \mu(B^\GG(z,r)) \leq \mu(B^\GG(z,tr)) \leq \mu(B^\GG(x,(t+1)r)).
        \]
        Combining this with Step 1, we get
        \begin{align*}
            \mu(E) &= \frac{1}{\omega_\GG (r/2)^Q} \int_{\GG} \mu(E \cap B^\GG(x,r/2)) \ dx \\
            &\leq \frac{(2(t+1))^Q}{\omega_\GG ((t+1)r)^Q} \int_{\GG} \frac{\mu(B^\GG(x,(t+1)r))}{\beta} \ dx \\
            &= \frac{(2(t+1))^Q}{\beta} \mu(K) < \varepsilon
        \end{align*}
        choosing $\beta > \frac{\mu(K)}{\varepsilon} (2(t+1))^d$ large enough. Notice that we have shown that the claim holds even for the supremum over all $r > 0$, and not just the $\limsup$, however at the start of the proof we assumed (without loss of generality) $\mu$ to be restricted to a compact set $K$. In the general case this may fail if $r$ becomes too large, hence the $\limsup$ as $r \downarrow 0$.
        
        \noindent\emph{Step 3}. From Step 2, we know that for any $\varepsilon > 0$ and any $k \geq 2$ there exist constants $\beta_k > 0$ and $\delta_k > 0$ such that
        \[
        \mu(\{ x \in K : \mu(B^\GG(x,kr)) \geq \beta_k \mu(B^\GG(x,r)) \}) \leq \frac{\varepsilon}{2^k} \qquad \text{whenever } r \in (0, \delta_k).
        \]
        Then for $r > 0$ call
        \[
        A_r := \left\{ x \in K : \exists k \geq 2 \text{ s. t. } r \in (0,\delta_k) \text{ and } \mu(B^\GG(x,kr)) \geq \beta_k \mu(B^\GG(x,r)) \right\}
        \]
        and notice that $\mu(A_r) \leq \varepsilon$ by subadditivity. Hence also
        \[
        A := \bigcup_{i=1}^\infty \bigcap_{j=i}^\infty A_{1/j}
        \]
        satisfies $\mu(A) \leq \varepsilon$ by continuity from below, so letting $\varepsilon \rightarrow 0$ this implies $\mu(A) = 0$.
        
        Take $x \in K \backslash A$. By construction for all $i \in \NN$ there exists $j \geq i$ such that $x \not\in A_{1/j}$, hence for all $k \in \NN$ such that $1/j < \delta_k$ we have
        \[
        \mu(B^\GG(x,k/j)) < \beta_k \mu(B^\GG(x,1/j)).
        \]
        Therefore for $\mu$-almost every $x_0 \in \supp(\mu)$ there exists a sequence $r^{x_0}_j = r_j \downarrow 0$ such that
        \[
        \limsup_{j \rightarrow \infty} \frac{\mu(B^\GG(x_0,kr_j))}{\mu(B^\GG(x_0,r_j))} \leq \beta_k \qquad \text{for all } k \in \NN.
        \]
        Thus the blow-up sequence $\nu_j = [\mu(B^\GG(x_0,r_j))]^{-1} T^{x_0, r_j}_\# \mu$ is such that, for every $k\in\NN$, the sequence $\nu_j(B^\GG(0,k))$ is bounded (in $j$), hence it is weakly-* precompact in $\mathcal{M}_\loc(\Omega;V)$. Up to a subsequence, it converges weakly-* to some Radon measure $\nu \in \mathcal{M}_\loc(\Omega;V)$, which is non-zero because
        \[
        \nu(\overline{B^\GG(0,1)}) \geq \limsup_{j \rightarrow \infty} \nu_j(\overline{B^\GG(0,1)}) = 1
        \]
        by weak-* upper semicontinuity on compact sets. Similarly, with balls of smaller radius one can show that $0 \in \supp \nu$, since $x_0 \in \supp \mu$.
    \end{proof}
    
    \bibliographystyle{siam}
    \bibliography{refs}
    
\end{document}